\documentclass[11pt]{article}

\usepackage{natbib,color,amssymb,amsmath,graphicx,float,booktabs,subfig,amsthm,hyperref,bm,multirow,threeparttable,textcomp}
\usepackage{JASA_manu}

\begin{document}
\theoremstyle{definition}
\newtheorem{assumption}{Assumption}
\newtheorem{theorem}{Theorem}
\newtheorem{lemma}{Lemma}
\newtheorem{example}{Example}
\newtheorem{definition}{Definition}
\newtheorem{corollary}{Corollary}
\newtheorem{proposition}{Proposition}
\newtheorem{condition}{Condition}

\renewcommand {\thecondition} {\Alph{condition}}

\newcommand{\red}{\color{red}}
\newcommand{\blue}{\color{blue}}
\newcommand{\mage}{\color{magenta}}

\def\logit{\text{logit}}

\title{\bf Identifiability of Normal and Normal Mixture Models With
Nonignorable Missing Data}
\author{
Wang Miao, Peng Ding, and Zhi Geng 
\footnote{ 
Wang Miao is Ph.D. Candidate (Email: \texttt{mwfy@pku.edu.cn}), Beijing International Center for Mathematical Research, Peking University, 5 Summer Palace Road 100871, Beijing, P. R. China.
Peng Ding is Ph.D. (E-mail: \texttt{pengdingpku@gmail.com}), Department of Statistics, Harvard University, 1 Oxford Street, Cambridge 02138, MA, USA.
Zhi Geng is Professor (E-mail: \texttt{zhigeng@pku.edu.cn}), School of Mathematical Sciences and Center for Statistical Science, Peking University, 5 Summer Palace Road 100871, Beijing, P. R. China.
Zhi Geng's research was supported by National Natural Science Foundation of China (11171365, 11021463), 973 Program (2015CB856000) and 863 Program (2015AA020507) of China.
The authors thank Mr. Zhichao Jiang of Peking University, Dr. Hua Chen of Institute of Applied Physics and Computational Mathematics in Beijing, Dr. Eric Tchetgen Tchetgen of Harvard T. H. Chan School of Public Health for helpful comments on early versions of our paper. We are grateful for the review comments from the Associate Editor and a reviewer.
}
}
\date{}
\maketitle

\begin{abstract}
Missing data problems arise in many applied research studies. They may jeopardize statistical inference of the model of interest, if the missing mechanism is nonignorable, that is, the missing mechanism depends on the missing values themselves even conditional on the observed data. With a nonignorable missing mechanism, the model of interest is often not identifiable without imposing further assumptions. We find that even if the missing mechanism has a known parametric form, the model is not identifiable without specifying a parametric outcome distribution. Although it is fundamental for valid statistical inference, identifiability under nonignorable missing mechanisms is not established for many commonly-used models. In this paper, we first demonstrate identifiability of the normal distribution under monotone missing mechanisms. We then extend it to the normal mixture and $t$ mixture models with non-monotone missing mechanisms. We discover that models under the Logistic missing mechanism are less identifiable than those under the Probit missing mechanism. We give necessary and sufficient conditions for identifiability of models under the Logistic missing mechanism, which sometimes can be checked in real data analysis. We illustrate our methods using a series of simulations, and apply them to a real-life dataset. 
\medskip

\noindent {\bfseries Keywords}: Heavy tail; Logistic model; Missing not at random; Monotone missing mechanism; Probit model; Selection model.
\end{abstract}

\section{\bf Introduction}
Missing data arise in many biomedical and socioeconomic studies. In the presence of missing data, the observed data may not be representative for the  population of interest, especially when the missing mechanism depends on the missing values themselves. For instance, if rich people tend not to respond in a survey, then the average wage obtained from the observed data will be lower than the truth.

An effective way to overcome this problem is to model the missing mechanism conditional on the observed covariates and the outcome. Using the terminologies in \cite{rubin1976inference}, the missing mechanism is called missing at random (MAR) if it does not depend on the missing values themselves conditional on the observed data, and it is called missing not at random (MNAR) otherwise. In the current literature, a variety of estimation methodologies based on the MAR assumption have been proposed, including likelihood-based inference, imputation, inverse probability weighting, and doubly robust methods. However, MNAR is often the case in practice, when the missingness depends on the missing values even conditional on the observed covariates. Unfortunately, statistical inference becomes quite challenging with data subject to MNAR mechanisms, because the models are often not identifiable based on the observed data. Many authors have studied models under MNAR mechanisms. Among them, the most popular approach is to model the conditional distribution of the missing indicator given the outcome and covariates, termed  the selection model \citep{little2002statistical}. \cite{greenlees1982imputation} propose maximum likelihood estimators for survey data with missing values, based on a fully parametric Logistic MNAR mechanism. \cite{qin2002estimation} propose an empirical likelihood estimation procedure for the case with nonparametric outcome model and parametric missing mechanism. \cite{ma2013semiparametric} study the semiparametric case with a symmetric outcome distribution and a parametric missing mechanism. Although they have developed useful estimation methods, the identifiability of their models may not be guaranteed even if the the missing mechanism is parametric. 
\cite{rotnitzky1998semiparametric} and \cite{scharfstein2003generalized} develop methods for conducting sensitivity analysis by assuming  completely specified outcome-dependent terms in the missing mechanisms. In his influential paper in econometrics, \cite{heckman1979sample} proposes the Heckman Selection Model consisting of an outcome equation and a selection equation indicating the latent variable for the missing mechanism. In the Heckman Selection Model, the missing mechanism is nonignorable when the bivariate normal error terms of the outcome equation and selection equation are correlated. As a fundamental problem, valid statistical inference of these methods relies on their identifiability. Unfortunately, identifiability does not always hold even in parametric models.

When an ``instrumental variable'' is available, i.e., there exists a variable that is associated with the outcome variable but independent of the missingness conditional on the outcome, identifiability of some MNAR models can be achieved. For example, \cite{chen2001parametric} shows identifiability of a subset of the regression parameters with biased sampling data, \cite{wang2014instrumental} show identifiability of certain nonparametric models with parametric missing mechanism, and \cite{chen2009identifiability} demonstrate the identifiability for binary outcomes with nonignorable missing data. However, it is not often feasible to find such an instrumental variable, without which the identifiability is not guaranteed in general.

In this paper, we focus on the identifiability of models under MNAR mechanisms. We illustrate with counterexamples the potential difficulty for achieving identifiability under nonparametric outcome models in Section 2. We prove identifiability of the normal model under the monotone missing mechanism in Section 3 without requiring any instrumental variables. We find that models under the frequently-used Logistic missing mechanism are less identifiable than those under the Probit missing mechanisms, and give necessary and sufficient conditions for the identifiability of models under the Logistic missing mechanism. In Section 4, we extend the results to normal mixture and $t$ mixture models, which are useful to accommodate more complex data features such as heavy-tailedness and multimodality. In Section 5, we propose a latent monotone missing mechanism and establish their identifiability. In Section 6, we evaluate the finite sample properties of the nonignorable missing data models via a series of simulations. The simulation results show advantages of normal mixture and $t$  mixture models for fitting complex data. In Section 7, we detect a nonignorable missing mechanism and latent components of the outcome variable in analysis of a real-life data on ambulatory expenditure. We conclude in Section 8, and relegate all the technical details to the Supplementary Materials.

\section{\bf Potential Difficulty for Nonparametric Identification}\label{sec::nonpara}

Throughout the paper, we let $X$ denote completely observed covariates, $Y$ denote the outcome variable, and $R$ denote the missing  indicator of $Y$ with $R=1$ if $Y$ is observed and $R=0$ otherwise. We use lower-case letters to denote realized values of the corresponding random variables, e.g., $y$ for a realized value of $Y$.  Suppose the observed data are $n$ independently and identically distributed samples, with some values of $Y$ missing.

The observed data allow us to identify only the observed distribution $P(y, R=1|x)$, which, however, does not suffice to determine the joint distribution $P(y,r|x)$ without additional assumptions. There are two equivalent ways to factorize the joint distribution:
$$
P(y, r|x)  =  P(y|x) P(r|x,y)  = P(r|x) P(y|x,r),
$$
with the first one being called the selection model and the second one being called the pattern mixture model \citep{little2002statistical}. Analogously, the observed distribution permits the following two equivalent factorizations:
$$
P(y, R=1|x) = P(y|x) P(R=1|x,y) = P(R=1|x) P(y|x,R=1).
$$
In this paper, we adopt the selection model factorization for our discussion. Fundamentally, we are interested in identifying the joint distribution $P(y, r|x) $ by the observed distribution $P(y, R=1|x)$. We say that the model is identifiable, if and only if the joint distribution $P(y, r|x) $ can be uniquely determined by the observed distribution $P(y, R=1|x)$, or, equivalently, two models yielding the same observed distribution must have the same joint distribution.

Because of their flexibility, nonparametric models are often used in the missing data literature \citep[e.g.,][]{qin2002estimation, ma2013semiparametric}. However, in general, the outcome distribution is not identifiable without specifying a parametric form for it, even if the missing mechanism is parametric. Below we provide three counterexamples to illustrate this potential difficulty.

\begin{example}\label{eg::unif}
Consider the following models:

\noindent Model \ref{eg::unif}.1. $Y\thicksim \text{Unif}(-0.5,0.5)$, and $P(R=1|y)=\Phi(y)$, where $\Phi(y)$ is the distribution function of the standard normal distribution;

\noindent Model \ref{eg::unif}.2. $Y$ has density $2\Phi(y)I\{-0.5\leq y \leq 0.5\}$, and $P(R=1|y)=\Phi(0)=1/2$.

The models above have the same parametric Probit missing mechanism, but they have different outcome distributions. Nevertheless, they have the same observed distribution because
$$
I\{-0.5\leq y \leq 0.5\} \cdot  \Phi(y)  = 2\Phi(y)I\{-0.5\leq y \leq 0.5\} \cdot  1/2.
$$
Therefore, if we assume a nonparametric model for the outcome with a Probit missing mechanism, they cannot be distinguished by the observed data $P(y, R=1)$. 
\end{example}

\begin{example}\label{eg::expo}
Consider the following models:

\noindent Model \ref{eg::expo}.1: $Y\sim \text{Exp}(2),$ $ \text{logit}\ P(R=1|y)=-\log 2 + y$;

\noindent Model \ref{eg::expo}.2: $Y\sim \text{Exp}(1),$ $ \text{logit}\ P(R=1|y) =\log 2 - y$.

The models above have the same parametric Logistic missing mechanism, but they have different outcome distributions. They have the same observed distribution because
$$
2e^{-2y} \cdot \frac{e^{-\log 2+y}}{1+e^{-\log 2 + y}}
=
e^{-y} \cdot \frac{e^{\log 2 - y}}{1+e^{\log 2-y}}.
$$
Therefore, they cannot be distinguished by the observed distribution.
\end{example}

The above counterexamples demonstrate the difficulty for obtaining identifiability of nonparametric outcome models even with parametric missing mechanisms. Unfortunately, even if we further restrict the outcome distributions to be symmetric as \cite{ma2013semiparametric} and assume parametric missing mechanisms, we still cannot identify the mean or distribution of the outcome.

\begin{example}\label{ex:logit}
Consider the following models:

\noindent Model \ref{ex:logit}.1. $Y\thicksim N(1,1)$, $\logit\ P(R=1|y)=-3/2+y$;

\noindent Model \ref{ex:logit}.2. $Y\thicksim N(2,1)$, $\logit\ P(R=1|y)=3/2-y$.

The models above have the same parametric Logistic missing mechanism, and symmetric outcome distributions. However, they cannot be identified by the observed distribution because
$$
\phi\left(y-1\right) \cdot  \frac{\exp(-3/2+y)}{1+\exp(-3/2+y)} = 
\phi\left(y-2\right) \cdot  \frac{\exp( 3/2-y)}{1+\exp( 3/2-y)} ,
$$
where $\phi(y)$ is the density of the standard normal distribution.
\end{example}

The above counterexamples show that nonparametric or semiparametric outcome models are generally not identifiable, even though the missing mechanisms are parametric. Without identifiability, the estimates for the parameters of the nonparametric outcome models may be misleading and of limited interest in practice.

\section{\bf Normal model with nonignorable missing data}\label{sec:2}
With nonignorable missing mechanisms, although identifiability of nonparametric outcome models is often hard to achieve as illustrated in Section \ref{sec::nonpara}, parametric outcome models are more likely to be identified. However, identifiability of many commonly-used parametric models have not been established in the literature. In this section, we discuss identifiability of the normal model with nonignorable missing data. We first show the conditions for its identifiability without covariates,
and illustrate the non-identifiability for some nonignorable missing mechanisms, such as the commonly-used Logistic missing mechanism. We then extend the result to nonignorable missing mechanisms with covariates, and utilize covariates to improve identifiability.

\subsection{Identifiability without covariates} \label{sec:2.1}

Suppose $Y\thicksim N(\mu, \sigma^2)$, with the following missing mechanism:
\begin{align}
P(R=1|y)=F(\alpha+\beta  y), \label{mis}
\end{align}
where $F(\cdot)$ is a known and strictly monotone distribution function with support on $(-\infty,+\infty)$. For instance, the standard normal distribution corresponds to the Probit missing mechanism, and the Logistic distribution corresponds to the Logistic missing mechanism. The missing mechanism is MAR if $\beta=0$, and it is MNAR if $\beta\neq 0.$ 
Mechanism (\ref{mis}) depicts that the response
probability is monotone in the value of the outcome.
For example, in some sensitive questionnaires,
people with higher outcome values tend not to respond, and therefore $\beta < 0$.

The following condition about the tail behavior of $F(\cdot)$ plays a central role in our discussion.

\begin{condition}\label{condi:A}
For any $\delta > 0$, $\lim_{z\rightarrow -\infty} F(z) / e^{\delta z}=0$ or
$+\infty$.
\end{condition}

Condition \ref{condi:A} requires the left tail decay rate of the response probability be not exponential. We can verify that the Probit missing mechanism satisfies Condition \ref{condi:A}, but the Logistic missing mechanism does not because $\lim_{z\rightarrow -\infty}  \{e^z/ (1+e^z) \}/e^z = 1$.

\begin{theorem}
\label{thm:NF} Suppose $Y\thicksim N(\mu, \sigma^2)$ with missing mechanism (\ref{mis}). Then
\begin{itemize}
\item[(a)] $\sigma^2$ and $|\beta|$ are identifiable;
\item[(b)]  $\mu$, $\sigma^2$, $\alpha$ and $\beta$ are identifiable,
if the sign of $\beta$ is known;
\item[(c)] $\mu$, $\sigma^2$, $\alpha$ and $\beta$ are identifiable,
if Condition \ref{condi:A} holds.
\end{itemize}
\end{theorem}

Result (a) above indicates that the variance of the outcome and the absolute value of its impact  on the response probability are always identifiable, but the sign of its impact may not be identifiable. Result (b)  implies that all parameters are identifiable if we have prior knowledge about the tendency of the missingness, i.e., the sign of $\beta$. Result (c) shows that all parameters are identifiable, if $F(\cdot)$ satisfies Condition \ref{condi:A}. Example \ref{ex:logit} in Section \ref{sec::nonpara} is also an example illustrating that the normal model under the Logistic missing mechanism may not be identifiable. A similar example is also presented in \cite{wang2014instrumental}.

If $F(\cdot)=T_\nu(\cdot)$,  the standard $t$ distribution function with degrees of freedom $\nu$,
then it satisfies Condition \ref{condi:A} when $\nu$ is known,
and thus all parameters are identifiable. The model with $F(\cdot)=T_\nu(\cdot)$ for the missing mechanism is sometimes referred to as the Robit model \citep{liu2004robit}. By varying the degrees of freedom $\nu$,
the Robit models can be used to approximate many missing mechanisms.
For instance, when $\nu$ tends to infinity,
it approximates the Probit model;
when $\nu$ is near $7$ or $8$, it approximates the Logistic model
\citep{mudholkar1978remark, liu2004robit}.
In fact, we have a stronger conclusion for the Robit model that the degrees of freedom parameter is also identifiable even if it is unknown. 

\begin{corollary}\label{cor:NT}
If $Y\thicksim N(\mu, \sigma^2)$, $P(R=1|y)=T_\nu(\alpha+\beta  y)$, and $\beta\neq0$,
then $\mu$, $\sigma^2$, $\nu$, $\alpha$ and $\beta$ are all identifiable.
\end{corollary}

\subsection{Identifiability with covariates} \label{sec:2.3}
When some completely observed covariates $X$ are available, we assume the outcome model
\begin{align}
Y|x  \thicksim &N\{   \mu(x,\gamma), \sigma^2(x,\theta) \},  \label{or}
\end{align}
and a generalized additive missing mechanism
\begin{align}
P(R=1|x,y)=F\{g(x,\alpha)+\beta y\}, \label{misx}
\end{align}
where $F(\cdot)$ is a known and strictly monotone distribution function,
$\mu(\cdot, \cdot)$ and $\sigma(\cdot, \cdot)$ have known forms, and
$(\gamma,\theta,\alpha,\beta)$ are unknown parameters. 
Note that  $(\gamma,\theta,\alpha)$ may be vectors. And we require that $\mu(\cdot,\gamma)$, $\sigma(\cdot,\theta)$ and $g(\cdot,\alpha)$ have a one-to-one mapping to $\gamma,\theta$ and $\alpha$, respectively. For instance, the linear function $\mu(x,\gamma)=x^T\gamma$ satisfies this condition.

\begin{theorem}\label{thm:NFx}
Assume that the outcome model is (\ref{or}) and the missing mechanism is (\ref{misx}). Then
\begin{itemize}
\item[(a)] $\theta$ and $|\beta|$ are identifiable;

\item[(b)] $\gamma$, $\theta$, $\alpha$ and $\beta$ are identifiable, if the sign of $\beta$ is known;

\item[(c)] $\gamma$, $\theta$, $\alpha$ and $\beta$ are identifiable, if Condition \ref{condi:A} holds;

\item[(d)]  $\gamma$, $\theta$, $\alpha$ and $\beta$ are identifiable, if the function pair $\{   \mu(\cdot, \gamma), \sigma^2(\cdot, \theta) \} $ and the function $g(\cdot, \alpha)$ are linearly uncorrelated, i.e., $a\cdot\mu(\cdot,\gamma)+b\cdot \sigma^2(\cdot,\theta)+g(\cdot,\alpha)\neq c$ for nonzero vector $(a,b,c)$ and for all $(\gamma, \theta, \alpha)$. 
\end{itemize}
\end{theorem}

In the above, (a), (b) and (c) are parallel with the results of Theorem \ref{thm:NF}.
Result (d) gives another sufficient condition for improving identifiability by observed covariates $X$, when Condition \ref{condi:A} does not hold. Results (d) does not allow the function $g(\cdot, \alpha)$ in the missing mechanism be linearly correlated with the function pair $\{   \mu(\cdot, \gamma), \sigma^2(\cdot, \theta) \} $, but it allows for dependence between the mean function $ \mu(\cdot, \gamma)$ and the variance function $\sigma^2(\cdot, \theta).$
We illustrate Theorem \ref{thm:NFx} by the following three examples.

\begin{example}\label{ex:probit}
Assume that
\[
Y|x  \thicksim N(\gamma_0+\gamma_1x, \sigma^2) ,\quad  P(R=1|x, y )=\Phi(\alpha_0+\alpha_1x+\beta y) .
\]
The Probit missing mechanism satisfies Condition \ref{condi:A} in Theorem \ref{thm:NFx}, and thus this model is identifiable. We can verify that this model is equivalent to the Heckman Selection Model \citep{heckman1979sample} up to different parametrizations. 
\end{example}

Contrary to the Probit missing mechanism, the Logisitic missing mechanism does not satisfy Condition \ref{condi:A} in Theorem \ref{thm:NFx} as pointed out before.
Many researchers \citep[e.g.,][]{greenlees1982imputation,glynn1986selection,qin2002estimation,robins1997analysis}
use the Logistic missing mechanisms when the outcome is MNAR.
However, it should be noted that
the Logistic missing mechanism is not identifiable even when some completely observed covariates are available,  as shown in the following example. 

\begin{example} \label{ex:logitx}
Assume that
$$
Y|x \thicksim N(\gamma_0+\gamma_1x, \sigma^2), \quad 
\logit~ P(R=1|x,y)=\alpha_0+\alpha_1x+\beta y. 
$$
Let $\sigma^2=1$ and $\gamma_1=0.5$.
Then the two different sets of parameters, $(\gamma_0,\alpha_0,\alpha_1,\beta)=(0,-2,-1,2)$ and $(2,2,1,-2)$,
lead to the same observed distribution. Therefore, the model above with the Logistic missing mechanism is not identifiable.
\end{example}

\begin{example} \label{ex:linear}
Assume that 
$$
Y|x  \thicksim N(\gamma_0+\gamma_1x, \sigma^2) ,\quad P(R=1|x,y)  = F(\alpha_0+\alpha_1 x + \alpha_2 x^2 + \beta y),
$$
where $\alpha_2 \neq 0.$ Without any assumptions on $F(\cdot)$, the missing mechanism satisfies the condition in Theorem \ref{thm:NFx}(d), where $X$ has a linear impact on $Y$ but a quadratic impact on its missingness.  Thus all parameters are identifiable, even though Condition \ref{condi:A} of Theorem \ref{thm:NFx} may not be satisfied. For instance, even if $F(\cdot)$ is the Logistic distribution function, the model is still identifiable with covariates $X$. Therefore, the non-identifiability problems in Example \ref{ex:logit} and \ref{ex:logitx} can be alleviated.
\end{example}

Note that in Example \ref{ex:logitx}, although $(\gamma_0,\alpha_0,\alpha_1,\beta)$ are not identifiable, the value of $\gamma_1$ is unique, which means that the slope of the outcome model is identifiable, but the intercept is not. We will show in the following theorem that this holds in general.

\begin{theorem}\label{thm:slope}
Assume that
\[
Y|x  \thicksim N\{  \mu(x,\gamma) ,  \sigma^2\}  ,  \quad   P(R=1|x, y)=F\{g(x,\alpha)+\beta y\}.
\]
Then the partial derivatives  $\partial \mu(x,\gamma)/\partial x$ and $|\partial g(x,\alpha)/ \partial x|$ are always identifiable.
\end{theorem}

Going back to the linear model $Y|x \thicksim N(\gamma_0+\gamma_1x ,  \sigma^2)$ in Example \ref{ex:linear}, Theorem \ref{thm:slope} says that even if $\alpha_2=0$, the coefficients of $X$ in the outcome model, $\gamma_1$, are always identifiable, although the intercept $\gamma_0$ may not be identifiable. The conclusion of Theorem \ref{thm:slope} can be further strengthened if $\mu(x, \gamma)$ and $g(x, \alpha)$ are both linear.

\begin{corollary}\label{cor:logit}
Assume that 
\[
Y|x \thicksim N( \gamma_0+x^T\gamma_1 ,  \sigma^2), \quad   
P(R=1|x, y)=F(\alpha_0+x^T\alpha_1+\beta y).
\]
Assume that Condition \ref{condi:A} fails, i.e., $\lim_{z\rightarrow -\infty}F(z)/e^{\delta z}=c$
for some $\delta>0$ and some $c\in(0,+\infty)$. 
Define $\tau_1=2\delta(\alpha_0+\beta\gamma_0)+\delta^2\sigma^2\beta^2+2\log(c)$ and $\tau_2=\alpha_1+\beta\gamma_1$. Then
\begin{enumerate}
\item[(a)] for a general distribution function $F(\cdot)$,  $\tau_1\neq 0$ or $\tau_2\neq 0$ is a sufficient condition for identifiability of  $(\gamma_0,\gamma_1,\sigma^2,\alpha_0,\alpha_1)$ and $\beta$;

\item[(b)] if $F(\cdot)$ is the Logistic distribution function, $\tau_1\neq 0$ or $\tau_2\neq 0$ is a necessary and sufficient condition for identifiability of 
$(\gamma_0,\gamma_1,\sigma^2,\alpha_0,\alpha_1)$ and $\beta$.

\end{enumerate}
\end{corollary}

The vector $\tau_2$ in the corollary above can be interpreted as the ``total impact'' of covariates $X$ on the missingness of the outcome $Y$. However, the scalar $\tau_1$ does not enjoy an apparent interpretation. From Corollary \ref{cor:logit}, the parameters are not identifiable only in a small subset of the parameter space, i.e., $\tau_1=0$ and $\tau_2=0$.  In fact, as shown in the proof of Corollary \ref{cor:logit}, even if $\tau_1=0$ and $\tau_2=0$, there are at most
two possible sets of parameters with the same observed distribution, and they must satisfy:
$\gamma_1'=\gamma_1, \beta'=-\beta,$ and $ \alpha_1'=-\alpha_1$.
In practice, we may select one of the parameter sets based on our domain knowledge. From these equations, we can see that the absolute values of the vector $\alpha_1$ are identifiable, implying that the absolute values of the coefficients of the covariates on the response probability are identifiable.  Moreover, we can see from the above discussion that the parameters $\tau_1$ and $\tau_2$ are always identifiable. Therefore, we can consistently estimate them and test whether they are equal to zero. As a result, we can test the necessary and sufficient condition in (b)
for the identifiability of the models based on the observed data.

\section{\bf Normal and \lowercase{$t$} mixtures with nonignorable missing data}\label{sec:MN}

In this section, we discuss the identifiability of the normal  mixture and $t$ mixture outcomes with nonignorable missing data.
These mixture distributions are useful to model outcomes with multiple modes or heavy tails.
We first consider the normal mixture model
\begin{gather}
Y\thicksim \sum_{k=1}^{K}\pi_k N(\mu_k,\sigma_k^2),\quad
\sum\limits_{k=1}^K \pi_k=1,\quad   
\pi_k \geq0,\quad 
k=1,\ldots,K.\label{otc:MN}
\end{gather}
The model above includes mixtures in both location and scale parameters. 
If $\mu_k=\mu$ for all $k$, it is a scale mixture; and if $\sigma_k=\sigma$ for all $k$, it is a location mixture. 
As in Section \ref{sec:2.1}, we assume that the missing mechanism is model (\ref{mis}), where
$F(\cdot)$ is a strictly monotone distribution function. It satisfies some of Condition \ref{condi:A} discussed before, and Conditions \ref{condi:B} and \ref{condi:C} below.

\begin{condition}
\label{condi:B}
For any $\theta_0 \in (-\infty, +\infty), \theta_1 >0, \delta_1,\delta_2\geq 0$, and $\delta_1+\delta_2\neq 0$, the limit $\lim_{z \rightarrow +\infty}\{F(\theta_0 +\theta_1 z)-F(y)\}/e^{-\delta_1 z^2-\delta_2 z}=0\text{ or }\infty$.
\end{condition}

\begin{condition}
\label{condi:C}
For any $\theta_0\in (-\infty, \infty), \theta_1>0,$ and $ M>0$, at most one of
 $
 \lim_{z\rightarrow+\infty}  z^M \{F(\theta_0 +\theta_1 z) - F(z)\}
 $
 and
 $
 \lim_{y\rightarrow-\infty}  z^M  \{1-F(\theta_0 +\theta_1 z)/F(z)\} 
 $
is finite and positive.
\end{condition}

Condition \ref{condi:B} requires the tail behavior of $ F(\alpha+\beta z)-F(z)$ be different from any normal or exponential density. It is straightforward to verify that the commonly-used Probit and Robit missing data mechanisms satisfy Condition \ref{condi:B}, but the Logistic missing mechanism does not. Condition \ref{condi:C} holds for all Logistic, Probit and Robit missing mechanisms.

We have the following result on the identifiability of normal mixture models.

\begin{theorem}\label{thm:MNF}
Suppose that the outcome model is a normal mixture (\ref{otc:MN}) with an unknown $K$ and $ \pi_k$'s, and  the missing mechanism is (\ref{mis}) with a known $F(\cdot)$ satisfying Conditions \ref{condi:A} and \ref{condi:B}.
Then all the parameters $K$, $\{ (\pi_k,\mu_k, \sigma_k^2): k=1,\ldots,K\}$, $\alpha$ and $\beta$ are identifiable.
\end{theorem}

Next we discuss identifiability of the $t$ model and its location mixture. A $t$ random variable can be represented by an infinite scale mixture of normal random variables \citep{little2002statistical, liu2004robit}, and thus the  mixture of $t$ random variables can be viewed as an infinite location and scale mixture of normal random variables. It can accommodate multimodality and heavy-tailedness simultaneously, and we can tune the degrees of freedom and the number of components in the mixture distribution in practice. We consider the outcome model
\begin{gather}
 Y \thicksim \sum_{k=1}^{K}\pi_k T_\nu(\mu_k,\omega^2),\quad
 \sum\limits_{k=1}^K \pi_k=1,\quad  
 \pi_k \geq0,\quad 
 k=1,\ldots,K,\label{otc:MT}
\end{gather}
where $T_\nu(\mu_k,\omega^2)$ is a $t$ random variable with location parameter $\mu_k$, scale parameter $\omega^2$, and degrees of freedom $\nu.$

\begin{theorem}\label{thm:TF}
Suppose that the outcome follows a location mixture of $t$ distributions in (\ref{otc:MT}) with unknown $K$, $ \pi_k$'s and $\nu$, and the missing mechanism is (\ref{mis}) with a known $F(\cdot)$ satisfying Condition \ref{condi:C}. Then all the parameters $K$, $\{( \pi_k,\mu_k): k=1,\ldots,K\} $, $\omega^2, \nu,\alpha$ and $\beta$ are identifiable.
\end{theorem}
As a special case of the $t$ mixture with $K=1$, the parameters are also identifiable if the outcome model follows a single $t$ distribution.

\section{\bf Latent monotone missing mechanism}

In this section, we still assume that the outcome $Y$ comes from a normal mixture distribution with $K$ components  indexed by parameters $(\mu_k,\sigma_k^2)$  and mixing proportions $\pi_k$ for $k=1, \ldots, K$. Note that $K$ and $ \pi_k$'s may be unknown. We further allow the missing mechanism to depend on the indicator, $G$, of the latent components of the mixture model
\begin{align}
P(R=1|y, G =k)=F(\alpha_k+\beta_k y ), \label{mis:LF} 
\end{align}
where $F(\cdot)$ is a known distribution function, $\alpha_k$ and $\beta_k$ are unknown parameters depending on latent groups. 
Unfortunately, this general missing mechanism is lack of identifiability. Even for the case with $\beta_k=0$, i.e., a latent ignorable missing mechanism \citep{frangakis1999addressing}, the model is not identifiable as illustrated in the following example.

\begin{example}
Assume that 
\begin{gather*}
Y\thicksim \pi_1  N( 0,  1)+ (1-\pi_1)  N(  0,  4),\quad 
P(R=1|y, G=k) = F(\alpha_k)\quad  (k=1,2).
\end{gather*}
Then any two models with $\pi_1 F(\alpha_1)=(1-\pi_1) F(\alpha_2)=1/4$ result in the same observed distribution.
\end{example}

Therefore, we should impose some restrictions on model (\ref{mis:LF}). As a motivating example, wage levels vary among different cities or communities, and in a survey an individual will decide to respond or not based on the average wage level of his city or community. In other words, it is the wage deviation of an individual from the city's average level that influences the response probability directly. The variance of the wages, a measure of the gap between the rich and the poor people, may also influence the response probability. Based on this intuition, we assume the following latent monotone missing mechanism, which depends on the mean and variance of the latent group:
\begin{gather}
P(R=1|y,G=k)=F\left[\frac{\alpha \psi(\sigma_k)+\beta \{y-\kappa(\mu_k)\}}{\varphi(\sigma_k)}\right], \label{mis:CondF}
\end{gather}
where $\psi(\cdot)$, $\kappa(\cdot)$ and $\varphi(\cdot)$ are known functions. We require that $\kappa(\cdot)$ and $\varphi(\cdot)$ be increasing, and  $\psi(\cdot)$ and $\varphi(\cdot)$ be positive. The common parameters $\alpha$ and $\beta$ represent similar missing data patterns in different latent groups. Given the component indicator of the mixture distribution, the missing mechanism (\ref{mis:LF}) is monotone in the outcome. Averaged over the latent components, however, the missing mechanism may not be monotone.

\begin{example}\label{eg::latent-monotone}
If $\psi(\sigma)=1, \varphi(\sigma) = 1$ and $\kappa(\mu)=0$, we have the monotone missing mechanism
$
P(R=1|y,G=k)=F(\alpha+\beta y),
$
as discussed in (\ref{mis}) and Sections \ref{sec:2} and \ref{sec:MN}.
\end{example}
\begin{example}
If $\psi(\sigma)=\sigma, \varphi(\sigma)=\sigma$ and $\kappa(\mu) = \mu$, we have
$$
P(R=1| y, G=k) = F\left(  \alpha + \beta \frac{y - \mu_k}{\sigma_k}  \right).
$$
The missing mechanism depends on the deviation from the center of each component $(y-\mu_k)/\sigma_k$, and the missing proportions are the same for all components.
\end{example}
\begin{example}
If $\psi(\sigma)=1, \varphi(\sigma)=\sigma$ and $\kappa(\mu) = \mu$, we have
$$
P(R=1| y, G=k) = F\left(  \frac{\alpha}{\sigma_k} + \beta \frac{y - \mu_k}{\sigma_k}  \right).
$$
The missing mechanism depends on the deviation from the center and the variance of each component, and the missing proportions are larger for components with larger variances. If $\mu_k = \mu$ for all components, then we have a scale mixture of normal distributions. The resulting marginal outcome distribution is a $t$ distribution, if as $K$ increases the empirical distribution of $\sigma_k$ approaches a scaled-inverse-$\chi^2$ distribution. We can verify that, under the above missing mechanism with $F(\cdot)=\Phi(\cdot)$, the model is equivalent to the selection-$t$ model \citep{marchenko2012heckman}.
\end{example}
\begin{example}
If $\psi(\sigma)=1, \varphi(\sigma)=\sigma^2$ and $\kappa(\mu) = \mu$, we have
$$
P(R=1| y, G=k) = F\left(  \alpha \sigma_k + \beta  \frac{ y - \mu_k }{\sigma_k}  \right).
$$
The missing mechanism depends on the deviation from the center and the variance of each component, and the missing proportions are smaller for components with larger variances.
\end{example}

In the above examples, we suggested several candidate missing mechanisms that are easy to interpret. Practitioners can choose their own missing mechanisms according to their background knowledge, as long as the functions $\{  \psi(\cdot), \varphi(\cdot), \kappa(\cdot) \} $ satisfy the conditions above. 
Below we show the identifiability of model (\ref{otc:MN}) under the missing mechanism (\ref{mis:CondF}).

\begin{theorem}\label{thm:MNCondF}
Suppose that the outcome model is a normal mixture (\ref{otc:MN}) with unknown $K$ and $ \pi_k$'s, and  the missing mechanism is (\ref{mis:CondF}). If $F(\cdot)$ is known and satisfies Conditions \ref{condi:A} and \ref{condi:B}.
then the parameters $K$, $ \{(\pi_k,\mu_k,\sigma_k^2): k=1,\ldots, K\}$, $\alpha$ and $\beta$ are identifiable.
\end{theorem}

The  missing mechanism (\ref{mis}) is a special case of the  missing mechanism (\ref{mis:CondF}) as shown in Example \ref{eg::latent-monotone}, and therefore Theorem \ref{thm:MNF} can be viewed as a special case of Theorem \ref{thm:MNCondF}. 
Similar to Section 2.2, we can incorporate covariates into  model (\ref{mis:CondF}), and establish identifiability results for the parameters conditional on the covariates.

\section{\bf Simulation Studies}

In this section, we use simulations to evaluate the finite sample properties of the models discussed in the previous sections. The data generating models are identifiable according to the theorems and corollaries in the previous sections.
For each model, we simulate $1000$ independent data sets under sample sizes $500$ and $1500$, and summarize the results with boxplots.

\subsection{Normal outcomes}
We generate the outcome variable $Y\thicksim N(0,4)$, and choose Probit, Logit, and Robit (with $\nu=2,16$) missing mechanisms.
We choose $\alpha=1$ and $\beta=1$ or $2$ for the missing mechanisms. 
For these settings, the missing data proportions  are between $30\%$ and $40\%$. For each data set, we apply various missing mechanisms to estimate the parameters, including Probit, Logistic, and Robit models with unknown degrees of freedom.

We show the simulation results only for $\beta=2$ in Figure \ref{fig:simu5.1}, and those for $\beta=1$ have similar patterns and are omitted. From Figure \ref{fig:simu5.1}, we can see that misspecification of the missing mechanism has little influence on estimation of the mean of the outcome, although it has some influence on the estimation of the missing mechanism itself.
It is because the quantiles of normal, Logistic, and $t$ distributions are almost linearly correlated over a large range \citep{mudholkar1978remark}. 
When the true missing mechanism is Robit with a small degrees of freedom and the Probit model is used for estimation,
it leads to small biases for estimating $\mu$, but large biases for estimating $\beta$. As the degrees of freedom increase, the biases for both $\mu$ and $\beta$ become smaller, and they can be improved by increasing the sample size.

The Probit model has explicit form of observed likelihoods, but those of the Logistic model and the Robit model involve integrals.  According to the simulation results, and for  computational convenience,
we recommend using the Probit missing mechanism for estimation of the mean of the outcome
if we do not care about the missing mechanism itself.
Otherwise, we recommend conducting sensitivity analysis
using the Robit missing mechanism with different degrees of freedom.

\subsection{Scale mixtures of normal and $t$ outcomes}\label{sec::scale-mixture-simu}
We use two-component scale mixtures of normal and $t$ distributions with $\nu=5$ for the outcome model. We choose both the Probit and latent Probit missing mechanisms.
We generate a covariate $X_1 \thicksim N(1,1)$, let $X=(1,X_1)^T$, and  then generate outcomes from the following models:
\begin{enumerate}
\item Scale  mixtures of normal outcome with the Probit missing mechanism
\[
Y|x  \thicksim \sum_{k=1}^2 \pi_k N(x^T\gamma,\sigma_k^2),\quad P(R=1|x,y)=\Phi\left(x^T\alpha+\beta y\right);
\]

\item Scale mixtures of normal outcome with the latent Probit missing mechanism
\[
Y|x  \thicksim \sum_{k=1}^2 \pi_k N(x^T\gamma,\sigma_k^2),\quad P(R=1|x,y,G=k)=\Phi\left\{\frac{x^T\alpha+\beta (y-x^T\gamma)}{\sigma_k}\right\};
\]

\item  $t$  outcome with the Probit missing mechanism
\[
Y|x,\sigma  \thicksim N(x^T\gamma,\sigma^2),\quad  \sigma^2\thicksim \frac{\omega^2\nu}{\chi^2_\nu},\quad P(R=1|x,y)=\Phi\left(x^T\alpha+\beta y\right);
\]

\item  $t$  outcome with the latent Probit missing mechanism (or, equivalently, the selection-$t$ model)
\begin{eqnarray}
\label{eq::re-para-selection-t}
Y|x,\sigma  \thicksim N(x^T\gamma,\sigma^2),\quad  
\sigma^2\thicksim \frac{\omega^2\nu}{\chi^2_\nu},\quad
P(R=1|x,y,\sigma)=\Phi\left\{\frac{x^T\alpha+\beta (y-x^T\gamma)}{\sigma}\right\}.
\end{eqnarray}
\end{enumerate}

The true parameters for  the above cases 1 and 2 
are  $\pi_1=\pi_2=0.5,\gamma=(1,1)^T,\sigma_1^2=1,\sigma_2^2=4$ and $\alpha=(1,1)^T$. The true parameters for the above cases 3 and 4
are $\nu=5$, $\gamma=(1,-1)^T,\omega^2=1$ and $\alpha=(1,1)^T$.
Further, we set $\beta =-0.5$ and $-1$, which result in about $15\%$ and $25\%$ missing values respectively.

We fit the generated datasets by three models: scale mixture of normal outcomes with correct missing mechanisms, normal outcome model with Probit missing mechanism, and the selection-$t$ model. We show the simulation results for $\beta=-0.5$  in Figure \ref{fig:simu5.2},  and those for $\beta=-1$ are similar and omitted. To save space, we present only the results for the coefficients of $X_1$ and $Y$.
From Figure \ref{fig:simu5.2}, when the missing mechanism is Probit (cases 1 and 3), the estimators of the three methods have small biases for outcome model parameter $\gamma$, but the estimators of the normal outcome model and the selection-$t$ model have large biases for missing mechanism parameters $(\alpha,\beta)$. 
This is because the former model cannot accommodate heavy-tailedness of the outcome distribution and the latter  is not a monotone missing mechanism. Using a scale mixture of normal outcome models, biases of both outcome model parameter $\gamma$ and missing mechanism parameters $(\alpha,\beta)$ become smaller as sample size increases (from 500 to 1500), but they  become even larger using the other two models.
When the missing mechanism is a latent Probit model (cases 2 and 4),
the selection-$t$ models are the true specifications for case 4.
Thus they have smallest biases for $\gamma,\alpha$ and $\beta$ overall.  
But a scale mixture of normal models also works very well for case 4, and their biases for $(\gamma,\alpha,\beta)$ are close to those of the selection-$t$ model.
However, for case 2, the selection-$t$ models have large biases for missing mechanism parameters $(\alpha,\beta)$, and they do not improve as sample size increases.
Among these estimation methods, the normal outcome model is the worst.
So we recommend a scale mixture of normal models for heavy-tailed outcomes because they enjoy robustness and easy interpretations.

\subsection{General mixtures of normal and location mixtures of $t$ outcomes}\label{sec::general-mixture-simu}
We generate data from two-component general mixtures of normal (mixture in both location and scale) outcome models with both the Probit and latent Probit missing mechanisms;  and two-component location mixture of  $t$  outcome models with the Probit missing mechanisms.
We first generate a covariate $X_1 \thicksim N(1,1)$, let $X=(1, X_1)^T$, and then generate the outcome according to the following models:
\begin{enumerate}
\item Normal mixture outcome with the Probit missing mechanism
\[
Y|x  \thicksim \sum_{k=1}^2 \pi_k N(x^T\gamma_k,\sigma_k^2),\quad P(R=1|x,y)=\Phi\left(x^T\alpha+\beta y\right),
\]
with true parameters  $\pi_1=\pi_2=0.5,\gamma_1=(1,1)^T,\theta=(1,-1)^T,\sigma_1^2=1,\sigma_2^2=4$ and $\alpha=(1,1)^T$.

\item Normal mixture outcome with the latent Probit missing mechanism
\[
Y|x  \thicksim \sum_{k=1}^2 \pi_k N(x^T\gamma_k,\sigma_k^2),\quad  P(R=1|x,y,G=k)=\Phi\left\{\frac{x^T\alpha+\beta (y-x^T\gamma_k)}{\sigma_k}\right\},
\]
with true parameters $\pi_1=\pi_2=0.5,\gamma_1=(2,1)^T,\theta=(-1.5,-2)^T,\sigma_1^2=1,\sigma_2^2=4$ and $\alpha=(1.5,-2)^T$.

\item  Location mixtures of $t$ outcome with the Probit missing mechanism
\[
Y|x, \sigma \thicksim\sum_{k=1}^2 \pi_k N(x^T\gamma_k,\sigma^2),\quad  \sigma^2\thicksim \frac{\omega^2\nu}{\chi^2_\nu},\quad P(R=1|x,y)=\Phi\left(x^T\alpha+\beta y\right),
\]
with true parameters  $\nu=5$, $\pi_1=\pi_2=0.5,\gamma_1=(1,3)^T,\theta=(1,-3)^T,\omega^2=1$ and $\alpha=(1,1)^T$.

\end{enumerate}
We choose $\beta=-0.5$ and $-1$ for missing mechanisms. Under these settings, the missing data proportions are between $20\%$ and $50\%$. Note that these three simulations are different from those in the previous subsection because the outcome models here allow for location mixtures, which usually lead to multimodality and have more parameters.

We fit the datasets by a two-component general mixture of normal outcome models with the Probit missing mechanism for cases 1 and 3, and a latent Probit missing mechanism for case 2. We also fit the datasets by the normal outcome with the Probit missing mechanism. Because the normal outcome model returns only one set of location parameters, we focus only on its estimators of missing mechanisms.

Figure \ref{fig:simu5.3} shows the simulation results for $\beta=-0.5$, and those for $\beta=-1$ are similar and omitted. To save space, we present only the results for the coefficients of $X_1$ and $Y$.
Using normal mixture models for estimation, the biases for parameters $(\gamma_1,\theta,\alpha,\beta)$ are small, and they become smaller as sample size increases. But the estimators by normal outcome model are very biased.
Therefore we recommend using a general mixture of normal models for outcomes with multimodality.

\section{\bf Application}

We apply our models to the ambulatory expenditure data previously analyzed by \cite{cameron2009microeconometrics}. 
The dataset contains the log of ambulatory expenditure (\texttt{ambexp}) and covariates $X=(1,\texttt{age, female, educ, blhisp, totchr, ins, income})$, including age, gender, education status, ethnicity, number of chronic diseases, insurance status, and income.
The dataset contains $3328$ observations and there are $526$ missing values of \texttt{ambexp}. More details about the data can be found in Chapter 16 of \cite{cameron2009microeconometrics}. The authors applied the Heckman Selection Model to the data, and they found no significant selection effect.

However,  we suspect that the samples do not exactly follow the normal model.
We apply the general normal mixture model to the data and compare the results with those of the Heckman Selection Model and the selection-$t$ model. 
We present the results in Table \ref{mus}. 
We start with two components and observe that the standard deviations of the two components, $\sigma_1$ and $\sigma_2$, are very different. The estimate of  the ratio $\sigma_1/\sigma_2$ is $0.610$ with $95\%$ confidence interval $(0.500, 0.716)$ excluding one. It provides strong empirical evidence of the existence of two latent normal components rather than a single one. We continue to fit a general normal mixture model with three components, and find that one component has estimated proportion $0.041$. We not only find that the proportion of the third component is small, but also find that the regression coefficients of the third component are similar to one component. Therefore, we choose a general normal mixture model with two components, and omit analysis with more components.

The general normal mixture model with the Probit missing mechanism yields interesting practical interpretations.
In our analysis, we first fit a general normal mixture model with different coefficients on all the covariates for the two components, and find that all regression coefficients are the same except for \texttt{ins}. Therefore, in the last two columns of Table \ref{mus}, we choose to present the results under a more parsimonious model allowing for only different regression coefficients of \texttt{ins}.
The covariate \texttt{ins} has a positive although insignificant effect  on $\log$(\texttt{ambexp}) in a large group of the people ($88\%$) with an estimate  $0.009$ and $95\%$ confidence interval $(-0.095,  0.174)$. 
However,   \texttt{ins} has  a significant negative effect in a small group ($12\%$) with an estimate $-0.595$ and $95\%$ confidence interval $(-1.705, -0.152)$.
The results provide further evidence of two different components. The Heckman Selection Model or the selection-$t$ model cannot detect such latent heterogeneity, because they are not able to accommodate location mixture or multimodality. As a result, these two models mix the two groups, and consequently report no significant impact of \texttt{ins} on $\log$(\texttt{ambexp}). 
The two components in the normal mixture model may reflect different health statuses. For healthy people who rarely needed health care, the insurance had an insignificant effect on ambulatory expenditure; for sick people who needed more health care, the insurance was beneficial for reducing ambulatory expenditure. The two components might also reflect different types of health insurance. 
Sick people tended to buy insurance that had larger coverage proportions, and consequently the insurance helped reduce more ambulatory expenditure.

Previous analyses found evidence of heavy-tailedness in the outcome distribution. However, the upper confidence limit of $\nu$ is the selection-$t$ model is near $20$ \citep{marchenko2012heckman, ding2014bayesian}, and such heavy-tailedness seems to be captured also by the normal mixture model as illustrated in the simulation studies in Sections \ref{sec::scale-mixture-simu} and \ref{sec::general-mixture-simu}. Furthermore, the Heckman Selection Model finds insignificant selection effect, but both the selection-$t$ and the normal mixture model with the Probit missing mechanism find significant selection effect.

% The $95\%$ confidence interval for the coefficient of  $\log$(\texttt{ambexp}) in the Probit missing mechanism, $\beta_Y$ in Table \ref{mus}, is $(-0.414, -0.025)$, which does not include zero.  

\begin{table}[ht]
\center
{\small 
\caption{The ambulatory expenditure study example. } \label{mus}
\begin{tabular}{lrlrlrl}
\toprule
&\multicolumn{2}{c}{Selection}&\multicolumn{2}{c}{Selection-$t$}& \multicolumn{2}{c}{General normal mixture}\\
\midrule
\multicolumn{2}{c}{Outcome Model}	&                                                                            \\
age    &  0.212 & (0.167, 0.257)    &0.207   & (0.163,  0.251)  & 0.209&(0.165,  0.253)\\
female &  0.348 & (0.230, 0.466)    &0.307   & (0.196,  0.417)  & 0.327&(0.213,  0.431)\\
educ   &  0.019 & (-0.002, 0.039)   &0.017   & (-0.003,  0.037) & 0.020&(0.000,  0.040)\\
blhisp &  -0.219& (-0.336, -0.102)  &-0.193  & (-0.306, -0.080)  &-0.211&(-0.326, -0.094)\\
totchr &  0.540 & (0.463, 0.617)    &0.513   & (0.443,  0.583)  & 0.529&(0.460,  0.588)\\
ins    &  -0.030& (-0.130, 0.070)   &-0.053  & (-0.151,  0.046)  & 0.009&(-0.095,  0.174)\\
       &        &                   &        &               &-0.595&(-1.705, -0.152)\\
$\sigma$& 1.271 &(1.236, 1.308)      &1.195& (1.146, 1.246)    & 1.159&(0.958,  1.227)\\
       &       &                   &        &                   & 1.900&(1.423,  2.383)\\
\midrule
\multicolumn{2}{c}{Selection Model}   &                                                                          \\
income &  0.003 &(0.000, 0.005)  &0.003  &(0.000,  0.006)        & 0.003&(0.000,  0.006)\\
age    &  0.088 &(0.034, 0.142)  &0.099  &(0.040,  0.157)       & 0.127&(0.060,  0.208)\\
female &  0.663 &(0.543, 0.782)  &0.725  &(0.591,  0.859)       & 0.734&(0.609,  0.911)\\
educ   &  0.062 &(0.038, 0.086)  &0.065  &(0.040,  0.090)       & 0.066&(0.045,  0.091)\\
blhisp &  -0.364&(-0.485, -0.243)&-0.394 &(-0.524, -0.263)     &-0.407&(-0.548, -0.291)\\
totchr &  0.797 &(0.658, 0.936)  &0.890  &(0.719,  1.061)    & 0.906&(0.750,  1.137)\\
ins    &  0.170 &(0.047, 0.293)  &0.180  &(0.048,  0.313)      & 0.166&(0.047,  0.289)\\
$\rho$&-0.131 &(-0.401; 0.161)&-0.322&(-0.526, -0.083)\\
$\beta_Y$&&&&& -0.174&(-0.414, -0.025)\\
\midrule
$ \pi_1$&& 1 && 1                                                  &  0.880&(0.401,  0.971)\\
$\log L$&\multicolumn{2}{c}{-5836.219}&\multicolumn{2}{c}{-5822.076}& \multicolumn{2}{c}{-5820.457}\\
AIC&\multicolumn{2}{c}{11706.44}&\multicolumn{2}{c}{11676.15} &\multicolumn{2}{c}{11680.91}\\
\bottomrule
\end{tabular}
\begin{flushleft}
Note: Point estimates and $95\%$ confidence intervals of the Heckman Selection Model (columns 2 and 3), the selection-$t$ model (columns 4 and 5), and the general mixtures of normal model (columns 6 and 7).
The $\rho$ parameter is only for the Heckman selection model and selection-$t$ model, and the $\beta_Y$ parameter is only for the general mixture of normal model. The $\pi_1$ parameter is the proportion of the first latent component, $\log L$ is the log likelihood of the observed data, and AIC is the Akaike Information Criterion.
\end{flushleft}
}
\end{table}

\section{\bf Discussion}

Dealing with missing data is crucial for many applied problems, which are often challenging if the missing mechanisms are nonignorable. Even if we have a fully parametric model on the missing mechanism, nonparametric outcome models subject to nonignorable missing data are not identifiable.  In this paper, we have demonstrated the identifiability of normal and normal mixture models with nonignorable missing data. The identifiability results do not require an instrumental variable for the missing mechanism, or the exclusion restriction assumption in the Heckman Selection Model \citep{wooldridge2010econometric}, i.e., there exists a covariate that is only in the selection equation but not in the outcome equation.
Although the selection-$t$ model could accommodate heavy-tailedness of the data, it could not model outcome with multimodality or latent groups. Our $t$ mixture model with a nonignorable missing mechanism filled
in the gap by allowing for modeling heavy-tailedness and multimodality simultaneously.

There are a few questions beyond the current scope of this paper. First, determining the number of components in the mixture models \citep{lo2001testing,chen2012inference} is an important problem in practice. 
Second, although the $t$ mixture model with the Robit missing mechanism, as discussed in Theorem \ref{thm:TF}, is useful for modeling heavy-tailed data, its maximum likelihood estimates involve numerical integrations. We are going to develop an efficient Bayesian inferential procedure for it.

\section*{\bf  Supplementary Materials}

Supplementary Materials contain the proofs of the theorems and corollaries.

\bibliographystyle{apa}
\bibliography{mixNormal}

\begin{figure}[hb]
\centering
\subfloat[Generated from the Probit missing mechanism]{
\includegraphics[width=.5\textwidth,height=0.3\textwidth]{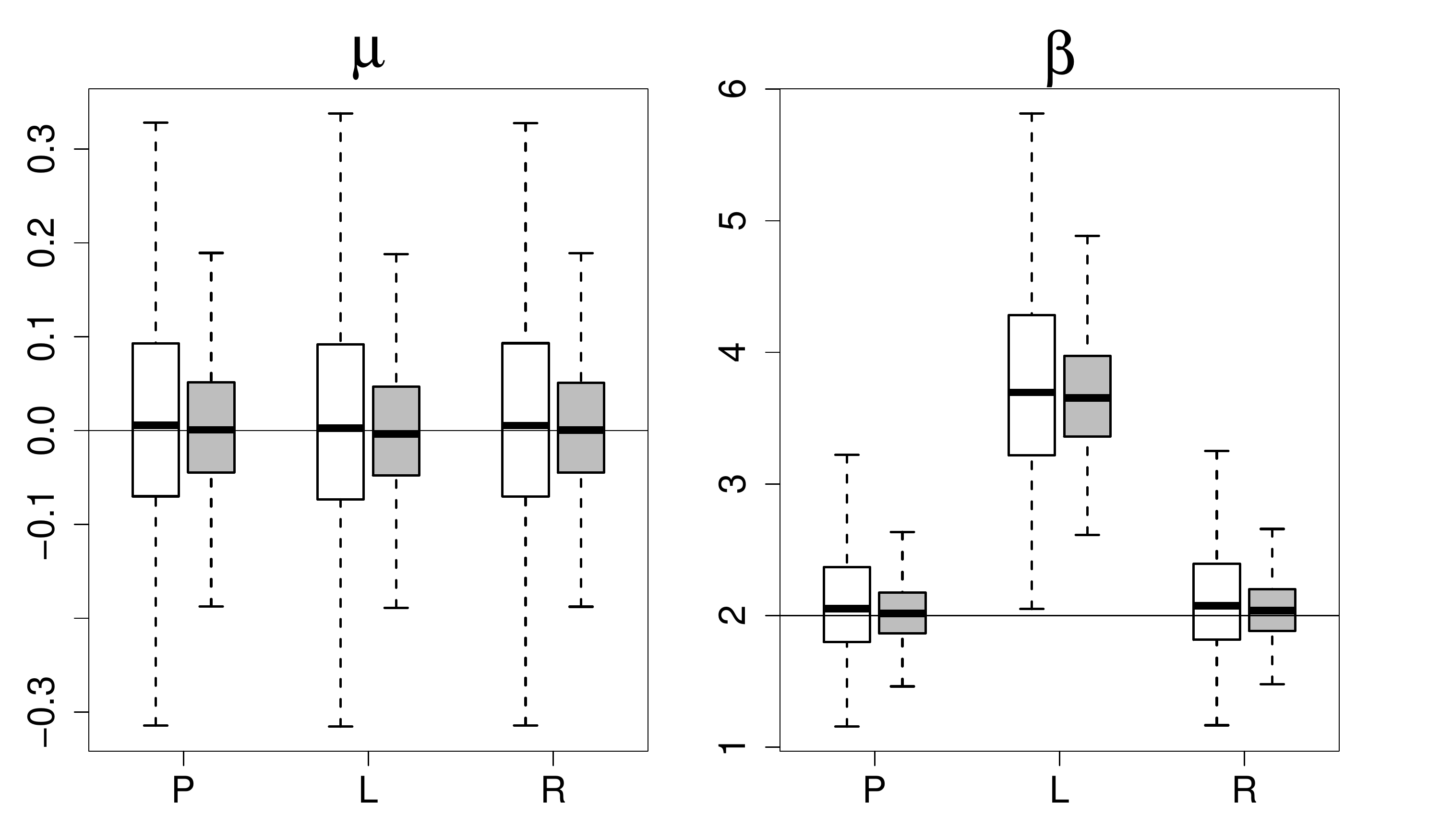}}\nolinebreak[4]
\subfloat[Generated from the Logistic missing mechanism]{
\includegraphics[width=.5\textwidth,height=0.3\textwidth]{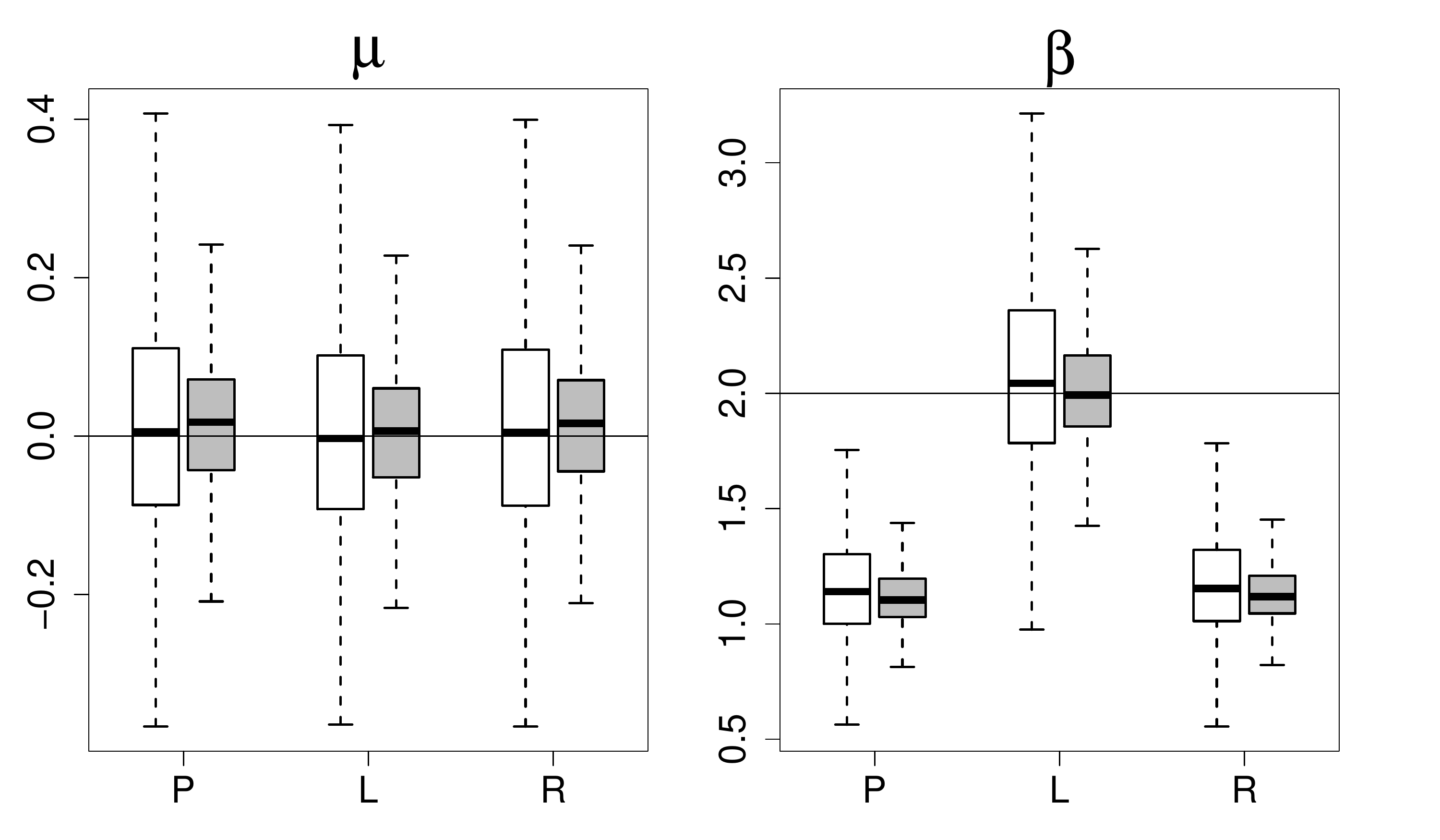}}\\
\subfloat[Generated from the Robit missing mechanism ($\nu=2$)]{
\includegraphics[width=.5\textwidth,height=0.3\textwidth]{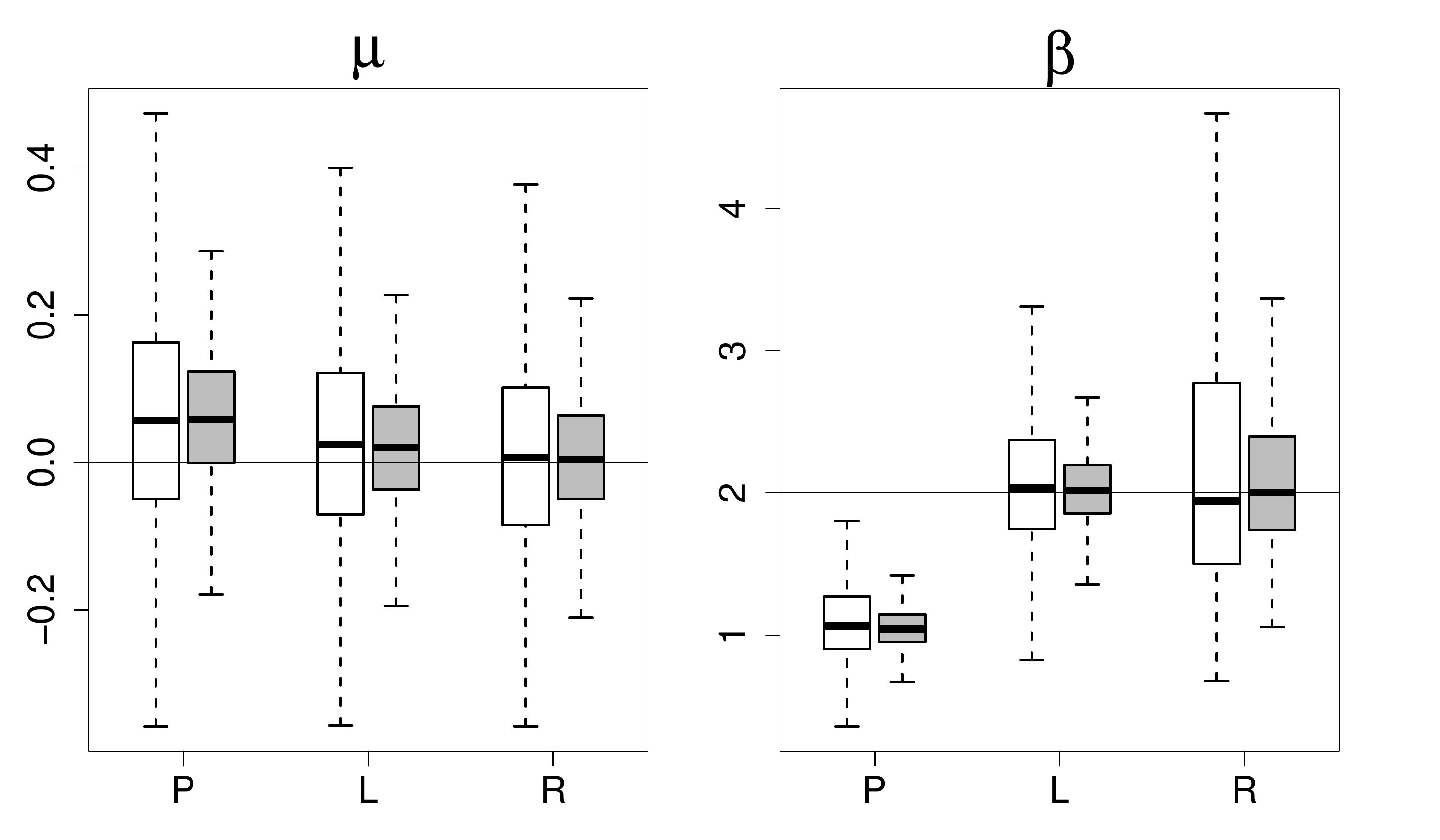}}\nolinebreak[4]
\subfloat[Generated from the Robit missing mechanism ($\nu=16$)]{
\includegraphics[width=.5\textwidth,height=0.3\textwidth]{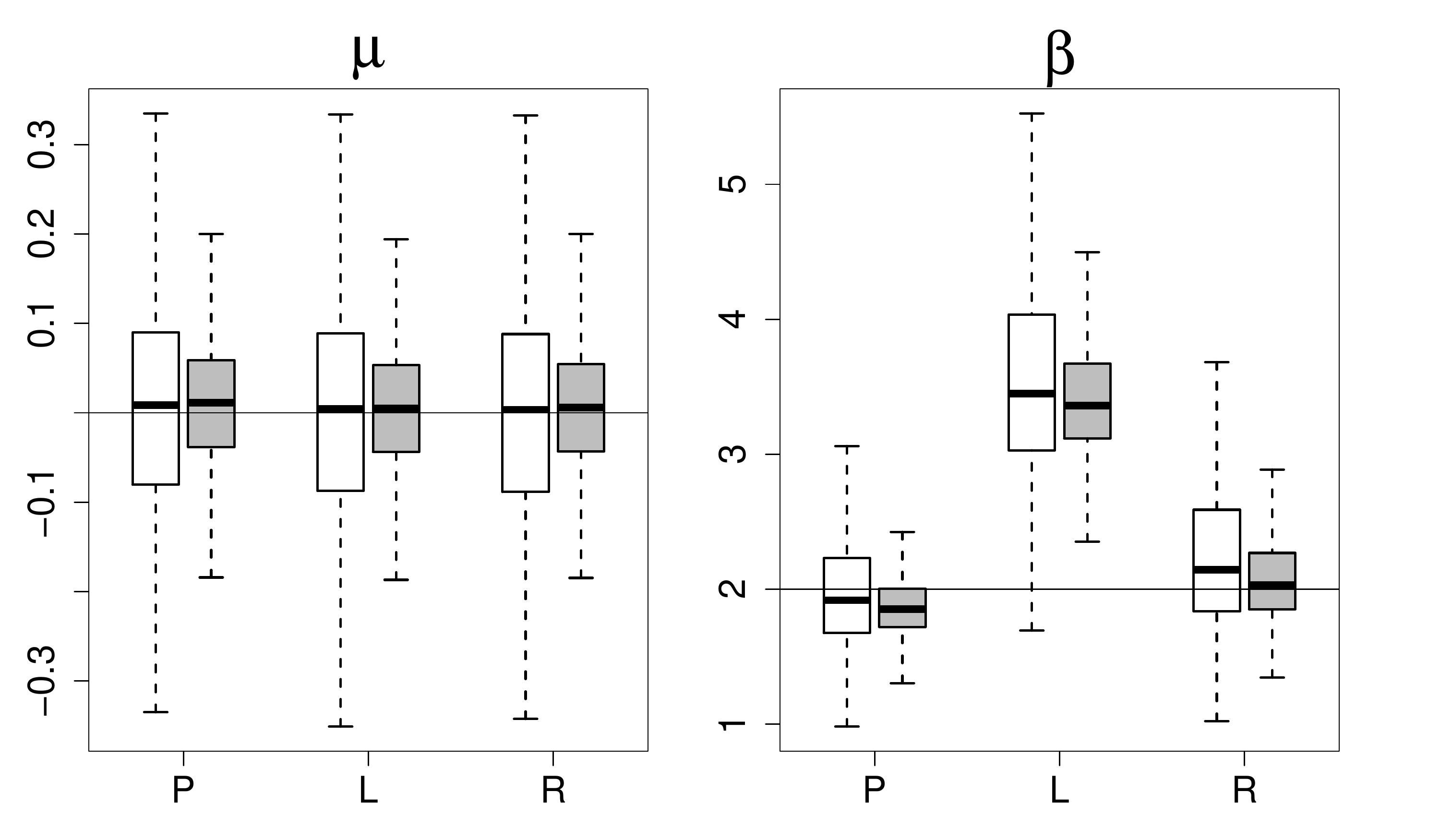}}\\
\caption{Estimates when $\beta=2$. Data are generated from normal outcome models with different missing mechanisms, and analyzed under normal outcome models with different missing mechanisms: Probit (P), Logistic (L) and Robit (R) with unknown degrees of freedom. In each boxplot, white boxes are for sample size 500 and gray boxes for 1500. The horizontal lines illustrate the true values of the parameters. } \label{fig:simu5.1}
\end{figure}

\begin{figure}
\centering
\subfloat[Generated from scale mixture of normal outcome with the Probit missing mechanism]{
\includegraphics[width=\textwidth,height=0.2\textwidth]{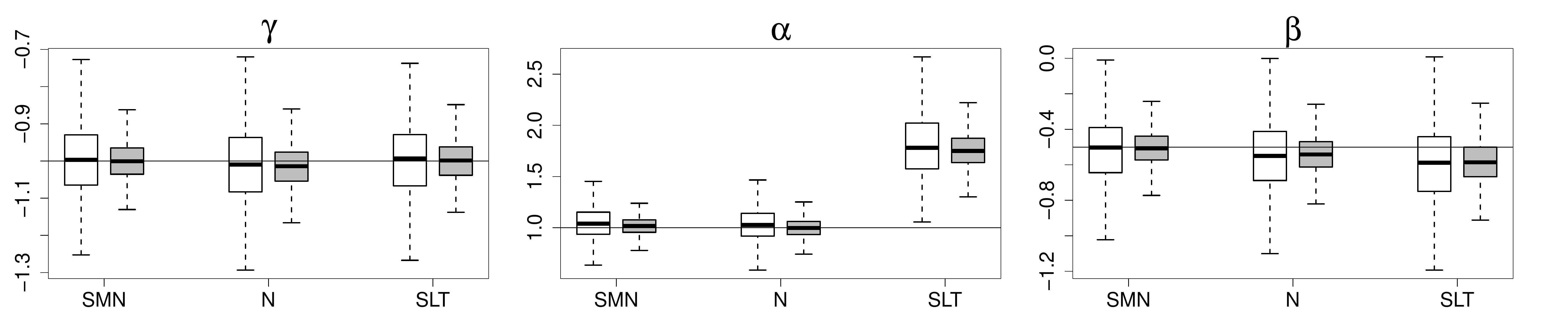}}\\
\subfloat[Generated from scale mixture of normal outcome with the latent Probit missing mechanism]{
\includegraphics[width=\textwidth,height=0.2\textwidth]{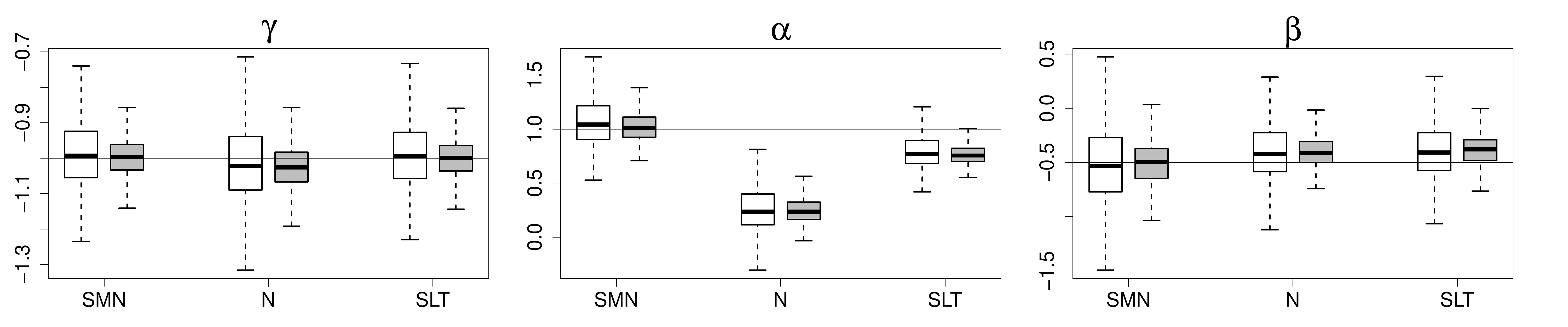}}\\
\subfloat[Generated from $t_{5}$ outcome with the Probit missing mechanism]{
\includegraphics[width=\textwidth,height=0.2\textwidth]{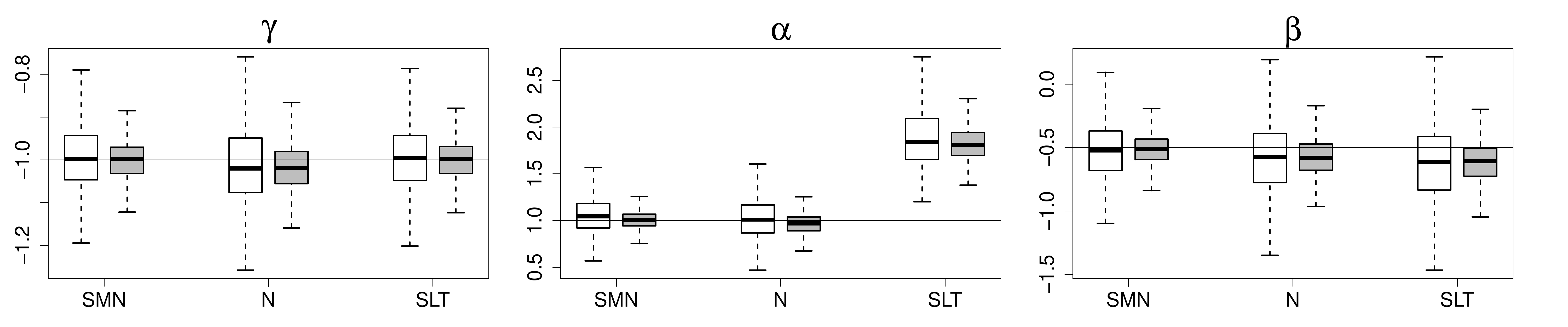}}\\
\subfloat[Generated from $t_{5}$ outcome with the latent Probit missing mechanism]{
\includegraphics[width=\textwidth,height=0.2\textwidth]{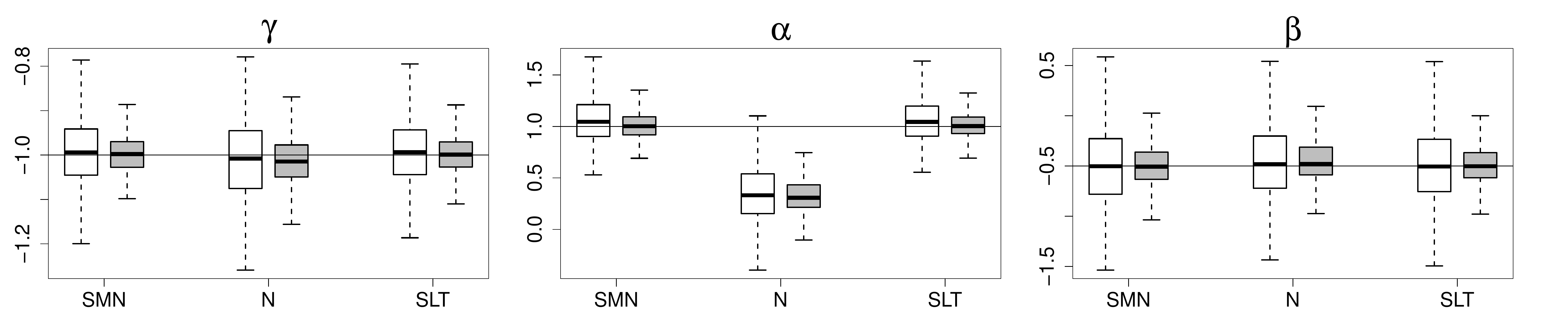}}
\caption{Estimates when $\beta=-0.5$. Data are generated from different outcome models, and analyzed under different outcome models with correct missing mechanisms: two-component scale mixture of normal outcome (SMN), normal outcome (N), and the selection-$t$ model (SLT). The estimators of SLT are re-parametrized to be equivalent to the scale mixture of normals with the latent Probit missing mechanism as shown in (\ref{eq::re-para-selection-t}).
In each boxplot, white boxes are for sample size 500 and gray boxes for 1500. The horizontal lines illustrate the true values of the parameters. } \label{fig:simu5.2}
\end{figure}

\begin{figure}
\centering
\subfloat[Generated from general mixture of normal outcome with the Probit missing mechanism]{
\includegraphics[width=\textwidth,height=0.3\textwidth]{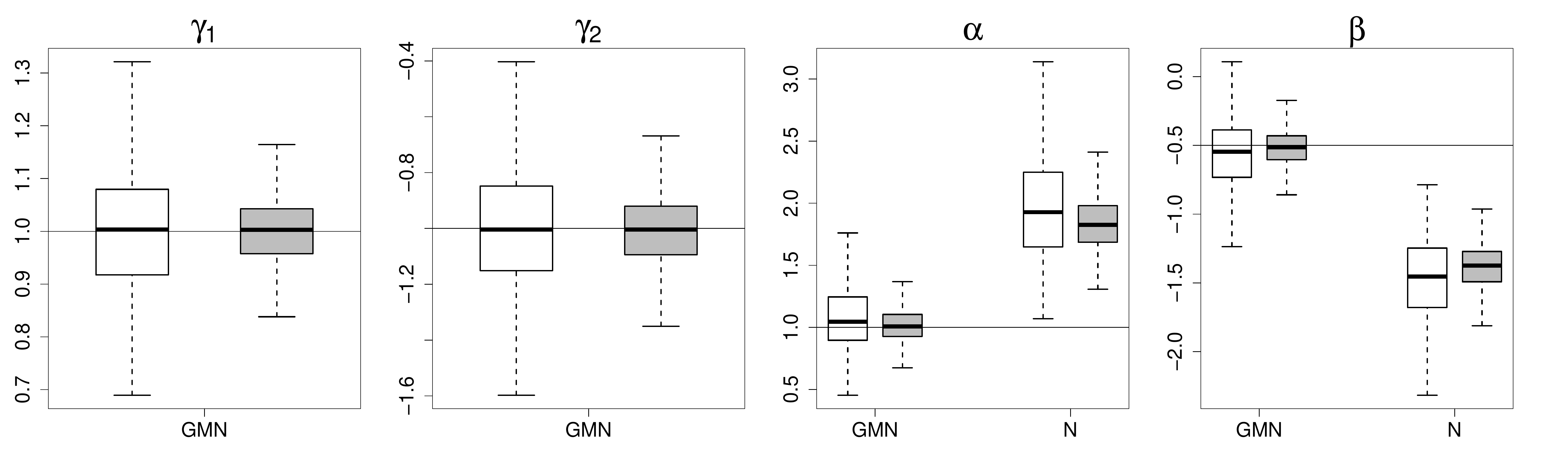}}\\
\subfloat[Generated from general mixture of normal outcome with the latent Probit missing mechanism]{
\includegraphics[width=\textwidth,height=0.3\textwidth]{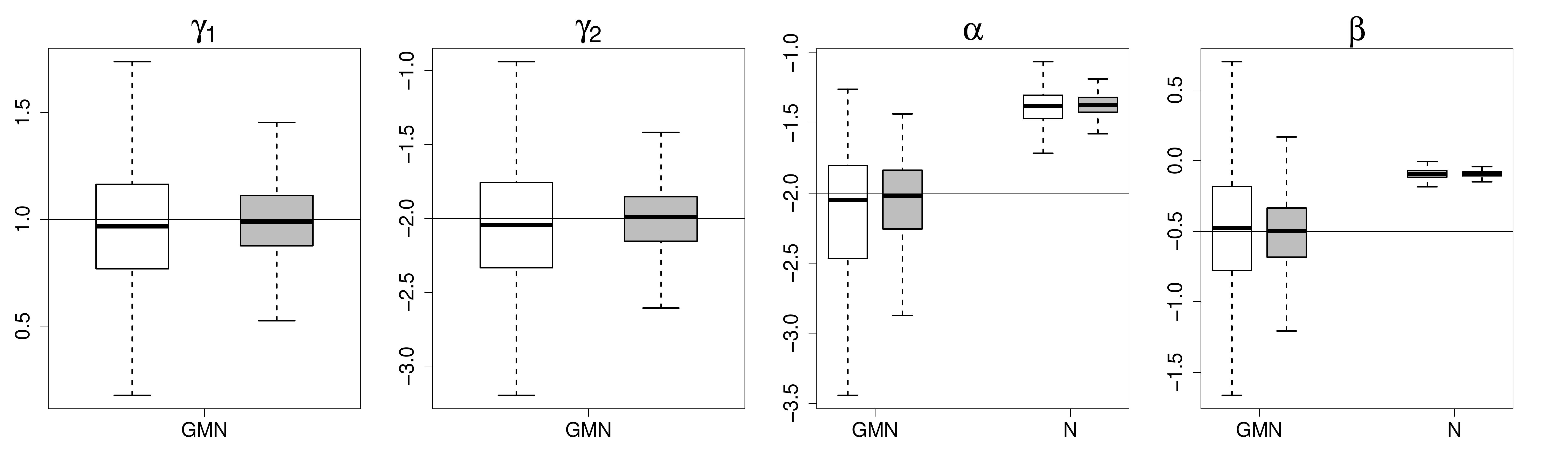}}\\
\subfloat[Generated from $t_{5}$ mixture outcome with the Probit missing mechanism]{
\includegraphics[width=\textwidth,height=0.3\textwidth]{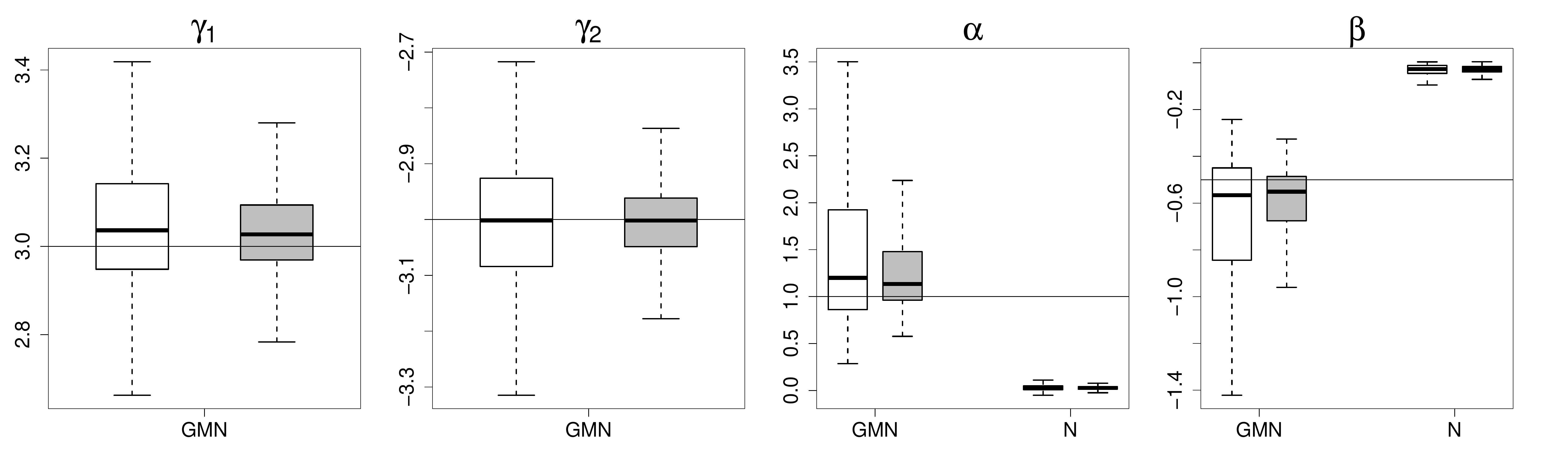}}
\caption{Estimates when $\beta=-0.5$. Data are generated from different outcome models, and analyzed under different outcome models with correct missing mechanisms: two-component general mixture of normal outcome (GMN), and normal outcome (N).
In each boxplot, white boxes are for sample size 500 and gray boxes for 1500. The horizontal lines illustrate the true values of the parameters. } \label{fig:simu5.3}
\end{figure}

\clearpage 
\appendix
\setcounter{page}{1}
\begin{center}
\title{\Large \bf Supplementary Materials for ``Identifiability of Normal and Normal Mixture Models With
Nonignorable Missing Data''}
\bigskip 

\date{}
\maketitle
\end{center}

\setcounter{equation}{0}
\setcounter{section}{0}
\setcounter{figure}{0}
\setcounter{example}{0}
\setcounter{proposition}{0}

\renewcommand {\theproposition} {A.\arabic{proposition}}
\renewcommand {\theexample} {A.\arabic{example}}
\renewcommand {\thefigure} {A.\arabic{figure}}
\renewcommand {\thetable} {A.\arabic{table}}
\renewcommand {\theequation} {A.\arabic{equation}}
\renewcommand {\thelemma} {A.\arabic{lemma}}
\makeatletter   %% HAVE TO ADD SOMETHING HERE TO MAKE IT SAY "APPENDIX"
 \renewcommand{\@seccntformat}[1]{APPENDIX~{\csname the#1\endcsname}.\hspace*{1em}}
 \makeatother

To prove the theorems and corollaries in this paper, we need several lemmas. 
\begin{lemma}\label{lem:1}
If $\sigma\neq\sigma'$, then for any $(\alpha,\beta,\alpha',\beta')$ and  distribution functions $F_1(\cdot)$ and $F_2(\cdot)$, at least one of the following two statements holds:
\[\lim\limits_{y\rightarrow+\infty}\frac{\phi\left(\frac{y-\mu}{\sigma}\right)}{\phi\left(\frac{y-\mu'}{\sigma'}\right)}\cdot\frac{F_1(\alpha+\beta y)}{F_2(\alpha'+\beta' y)}=+\infty\text{ or } 0\text{ for any $\mu,\mu'$;}\]
\[\lim\limits_{y\rightarrow-\infty}\frac{\phi\left(\frac{y-\mu}{\sigma}\right)}{\phi\left(\frac{y-\mu'}{\sigma'}\right)}\cdot\frac{F_1(\alpha+\beta y)}{F_2(\alpha'+\beta' y)}=+\infty\text{ or } 0\text{ for any $\mu,\mu'$.}\]
\end{lemma}

\begin{proof}
We first simplify the expression as
\begin{gather}
\frac{\phi\left(\frac{y-\mu}{\sigma}\right)}{\phi\left(\frac{y-\mu'}{\sigma'}\right)}\cdot\frac{F_1(\alpha+\beta y)}{F_2(\alpha'+\beta' y)}=\exp\left\{\frac{(\sigma^2-\sigma'^2)y^2}{2\sigma^2\sigma'^2}+\frac{(\sigma'^2\mu-\sigma^2\mu')y}{\sigma^2\sigma'^2}+\frac{\sigma^2\mu'^2-\sigma'^2\mu^2}{2\sigma^2\sigma'^2}\right\}\cdot\frac{F_1(\alpha+\beta y)}{F_2(\alpha'+\beta' y)},\label{prf:lem1}
\end{gather}
and then discuss its limit for the following four cases.

\begin{enumerate}
\item[(a)] $\beta\geq0,\beta'\geq0$. Because $\lim_{y\rightarrow+\infty}F_1(\alpha+\beta y)/F_2(\alpha'+\beta' y)$ is finite and positive, as $y\rightarrow+\infty$ the limit of expression (\ref{prf:lem1}) is $+\infty$ if $\sigma>\sigma'$ and $0$ if $\sigma<\sigma'$, for any $\mu,\mu'$.

\item[(b)] $\beta<0,\beta'<0$. Because $\lim_{y\rightarrow-\infty}F_1(\alpha+\beta y)/F_2(\alpha'+\beta' y)=1$, as $y\rightarrow-\infty$ the limit of expression (\ref{prf:lem1}) is $+\infty$ if $\sigma>\sigma'$ and $0$ if $\sigma<\sigma'$, for any $\mu,\mu'$.

\item[(c)] $\beta\geq0,\beta'<0$. Because $\lim_{y\rightarrow+\infty}F_1(\alpha+\beta y) $ is finite and positive and $\lim_{y\rightarrow-\infty}F_2(\alpha'+\beta' y)=1$, as $y\rightarrow+\infty$ the limit of expression (\ref{prf:lem1}) is $+\infty$ for $\sigma>\sigma'$, and as $y\rightarrow-\infty$ its limit is $0$ for $\sigma<\sigma'$, for any $\mu,\mu'$.

\item[(d)] $\beta<0,\beta'\geq0$. Because $\lim_{y\rightarrow-\infty}F_1(\alpha+\beta y)=1$ and $\lim_{y\rightarrow+\infty}F_2(\alpha'+\beta' y)$ is finite and positive, as $y\rightarrow-\infty$ the limit of expression (\ref{prf:lem1}) is $+\infty$ for $\sigma>\sigma'$, and as $y\rightarrow+\infty$ its limit is $0$ for $\sigma < \sigma'$, for any $\mu,\mu'$.
\end{enumerate}
Therefore, at least one of the two statements holds.
\end{proof}

\begin{lemma}\label{lem:2}
If $\beta\geq0$, $\beta'\geq0$, then for any $\mu\neq\mu'$ and any distribution functions $F_1(\cdot)$ and $F_2(\cdot)$, we have
\[\lim_{y\rightarrow+\infty}\frac{\phi\left(\frac{y-\mu}{\sigma}\right)}{\phi\left(\frac{y-\mu'}{\sigma}\right)}\cdot\frac{F_1(\alpha+\beta y)}{F_2(\alpha'+\beta' y)}=0\text{ or }+\infty.\]
If $\beta < 0, \beta' < 0$, we have the same result as $y\rightarrow -\infty$.

\end{lemma}

\begin{proof}
We simplify the expression as
\begin{gather}
\frac{\phi\left(\frac{y-\mu}{\sigma}\right)}{\phi\left(\frac{y-\mu'}{\sigma}\right)}\cdot\frac{F_1(\alpha+\beta y)}{F_2(\alpha'+\beta' y)}=\exp\left\{\frac{(\mu-\mu')y}{\sigma^2}+\frac{\mu'^2-\mu^2}{2\sigma^2}\right\}\cdot\frac{F_1(\alpha+\beta y)}{F_2(\alpha'+\beta' y)}.\label{prf:lem2}
\end{gather}
If $\beta\geq0$ and $\beta'\geq0$, $\lim_{y\rightarrow+\infty}F_1(\alpha+\beta y)/F_2(\alpha'+\beta' y)$ is finite and positive. As $y \rightarrow +\infty$, if $\mu>\mu'$, the limit of expression (\ref{prf:lem2}) is $+\infty$; if $\mu<\mu'$, its limit is $0$. 
If $\beta<0$ and $\beta' < 0$, we let  $\tilde \beta=-\beta$, $\tilde \beta'=-\beta'$, $\tilde y=-y$, $\tilde\mu=-\mu$ and  $\tilde\mu'=-\mu'$. We have the same result as $y\rightarrow -\infty$, or equivalently $\tilde y \rightarrow +\infty$, by the same argument.
\end{proof}

\begin{lemma}\label{lem:3}
If $\lim_{z\rightarrow-\infty}F_i(z)/e^{\delta z}=0\text{ or }+\infty$ for any $\delta>0$ ($i=1,2$), then for any $(\alpha,\beta,\alpha',\beta',\sigma)$ at least one of the following statements holds:
\[\lim\limits_{y\rightarrow+\infty}\frac{\phi\left(\frac{y-\mu}{\sigma}\right)}{\phi\left(\frac{y-\mu'}{\sigma}\right)}\cdot\frac{F_1(\alpha+\beta y)}{F_2(\alpha'+\beta' y)}=0\text{ or }+\infty \text{ for any $\mu\neq\mu'$;}\]
\[\lim\limits_{y\rightarrow-\infty}\frac{\phi\left(\frac{y-\mu}{\sigma}\right)}{\phi\left(\frac{y-\mu'}{\sigma}\right)}\cdot\frac{F_1(\alpha+\beta y)}{F_2(\alpha'+\beta' y)}=0\text{ or }+\infty \text{ for any $\mu\neq\mu'$.}\]
\end{lemma}

\begin{proof}
We first simplify the expression as
\begin{gather}
\frac{\phi\left(\frac{y-\mu}{\sigma}\right)}{\phi\left(\frac{y-\mu'}{\sigma}\right)}\cdot\frac{F_1(\alpha+\beta y)}{F_2(\alpha'+\beta' y)}=\exp\left\{\frac{(\mu-\mu')y}{\sigma^2}+\frac{\mu'^2-\mu^2}{2\sigma^2}\right\}\cdot\frac{F_1(\alpha+\beta y)}{F_2(\alpha'+\beta' y)} ,\label{prf:lem3}
\end{gather}
and then discuss its limit for the following four cases.
\begin{enumerate}
\item[(a)] $\beta\geq0$, $\beta'\geq0$. We have proved this case in Lemma \ref{lem:2}.

\item[(b)] $\beta<0$, $\beta'<0$. Letting $\tilde\beta=-\beta$, $\tilde\beta'=-\beta'$, $\tilde y=-y$, $\tilde\mu=-\mu$ and $\tilde \mu'=-\mu'$, we have $\tilde\beta>0$, $\tilde\beta'>0$ and
    \[\frac{\phi\left(\frac{y-\mu}{\sigma}\right)}{\phi\left(\frac{y-\mu'}{\sigma}\right)}\cdot\frac{F_1(\alpha+\beta y)}{F_2(\alpha'+\beta' y)}=\frac{\phi\left(\frac{\tilde y-\tilde \mu}{\sigma}\right)}{\phi\left(\frac{\tilde y-\tilde \mu'}{\sigma}\right)}\cdot\frac{F_1(\alpha+\tilde \beta \tilde y)}{F_2(\alpha'+\tilde \beta' \tilde y)}.\]
    By Lemma \ref{lem:2}, its limit is either $0$ or $+\infty$ as $\tilde y\rightarrow +\infty$, or equivalently $y\rightarrow-\infty$, for any $\mu\neq\mu'$.
    
\item[(c)] $\beta\geq0$, $\beta'<0$. 
First, we discuss the case with $\mu>\mu'$. 
We have $\lim_{y\rightarrow+\infty}F_2(\alpha'+\beta' y)=0$, $\lim_{y\rightarrow+\infty}F_1(\alpha+\beta y)$ is finite and positive, and $\lim_{y\rightarrow+\infty}\exp\left\{(\mu-\mu')y/\sigma^2\right\}=+\infty$. Therefore, the limit of (\ref{prf:lem3})
is $+\infty$. Second, we discuss the case with $\mu < \mu'$. 
As $y\rightarrow+\infty$, $z=\alpha'+\beta' y\rightarrow-\infty$, because $\lim_{z\rightarrow-\infty}F_2(z)/e^{\delta z}=0\text{ or }+\infty$ with $\delta=(\mu-\mu')/(\sigma^2\beta')>0$, the limit of (\ref{prf:lem3}) is $0$ or $+\infty$.

\item[(d)] $\beta<0$, $\beta'\geq0$. Exchange numerator and denominator in (\ref{prf:lem3}), and we can follow (c) to prove this case.
\end{enumerate}
\end{proof}

\begin{lemma}\label{lem:4}
For any distribution function $F(\cdot)$, the limit
\begin{gather*}
\lim\limits_{y\rightarrow+\infty}\frac{\phi\left(\frac{y-\mu}{\sigma}\right)}{\phi\left(\frac{y-\mu'}{\sigma}\right)}\cdot\frac{1}{F(\alpha+\beta' y)}
\end{gather*}
is finite and positive, if and only if 
\begin{enumerate}
\item[(a)]
$\beta' \geq 0$ and $\mu'=\mu$; or
\item[(b)] 
$\beta' < 0$, $\mu' > \mu$, and $\lim_{z\rightarrow-\infty}F(z)/e^{\delta z}=c\in(0,+\infty)$, where $\delta=(\mu-\mu')/(\sigma^2\beta')$, and $c=\exp\{(\mu'^2-\mu^2)/(2\sigma^2)-\delta\alpha\}$. 
\end{enumerate} 
\end{lemma}
\begin{proof}
We have
\[\lim\limits_{y\rightarrow+\infty}\frac{\phi\left(\frac{y-\mu}{\sigma}\right)}{\phi\left(\frac{y-\mu'}{\sigma}\right)}\cdot\frac{1}{F(\alpha+\beta' y)}=\exp\left\{\frac{\mu'^2-\mu^2}{2\sigma^2}\right\}\cdot\lim\limits_{y\rightarrow+\infty}\frac{\exp\left\{\frac{(\mu-\mu')y}{\sigma^2}\right\}}{F(\alpha+\beta' y)}.\]
For the case with $\beta' \geq0$, the limit is finite and positive if and only if $\mu'=\mu$. For the case with $\beta' <0$, the limit is finite and positive if and only if $\lim_{z\rightarrow-\infty}F(z)/e^{\delta z}=c\in(0,+\infty)$, where
$$
\delta=\frac{\mu-\mu'}{\sigma^2\beta'}>0, \quad c=\exp\left(\frac{\mu'^2-\mu^2}{2\sigma^2}-\delta\alpha\right), \quad 
z=\alpha+\beta' y \rightarrow - \infty.
$$
\end{proof}

\begin{lemma}\label{lem:5}
For any positive integer $K=\sum_{i=1}^I J_i$ and any parameters $\{ (\sigma_i^2,\mu_{ij}): i=1,\ldots,I; j=1,\ldots,J_i \} $, $\alpha$ and $\beta$, if $\sigma_{i+1}^2<\sigma_i^2$ and $\mu_{i(j+1)}<\mu_{ij}$, then the functions 
$$
\left\{ Q_{ij}(y)=\phi\left(   \frac{ y-\mu_{ij}}{ \sigma_i}  \right)  
F\left[     \frac{  \alpha \psi(\sigma_k)+\beta\{y-\kappa(\mu_{ij}) \}   } { \varphi(\sigma_i) } \right]  : i=1, \ldots, I; j=1,\ldots, J_i \right\}
$$ 
are linearly uncorrelated. 
\end{lemma}

\begin{proof}
Suppose there were real numbers $\{p_{ij}:i=1,\ldots,I; j=1,\ldots,J_i\}$ such that
\begin{gather}
\sum\limits_{i=1}^I\sum\limits_{j=1}^{J_i} p_{ij} Q_{ij}(y)=0,\label{eq::mixture-lemma5}
\end{gather}
for any $y.$ We will discuss Equation (\ref{eq::mixture-lemma5}) for the following two cases. 
\begin{enumerate}
\item[(a)] $\beta\geq0$.
Dividing Equation (\ref{eq::mixture-lemma5}) by $\phi\{  (y-\mu_{11}) / \sigma_1   \} / \sigma_1$, we have 
$
S_1(y)+S_2(y)+S_3(y)=0 ,
$
for any $y,$ where
\begin{align*}
S_1(y)=&p_{11}F\left[\frac{\alpha\psi(\sigma_i)+\beta\{y-\kappa(\mu_{11})\}}{\varphi(\sigma_1)}\right],\\
S_2(y)=&\sum\limits_{j=2}^{J_1}\frac{p_{1j}\sigma_1}{\sigma_1}\exp\left\{\frac{(\mu_{1j}-\mu_{11})y}{\sigma_1^2}+\frac{\mu_{11}^2-\mu_{1j}^2}{2\sigma_1^2}\right\}F\left[\frac{\alpha\psi(\sigma_i)+\beta\{y-\kappa(\mu_{ij})\}}{\varphi(\sigma_1)}\right],\\
S_3(y)=&\sum\limits_{i=2}^I\sum\limits_{j=1}^{J_i}\frac{ p_{ij}\sigma_1}{\sigma_i}\exp\left\{\frac{(y-\mu_{11})^2}{2\sigma_1^2}-\frac{(y-\mu_{ij})^2}{2\sigma_i^2}\right\}F\left[\frac{\alpha\psi(\sigma_i)+\beta\{y-\kappa(\mu_{ij})\}}{\varphi(\sigma_i)}\right],
\end{align*}
correspond to the first term, the sum of the second to the $J_1$th terms with variance $\sigma_1^2$, and the sum of the remaining terms of Equation (\ref{eq::mixture-lemma5}), respectively.
We can verify that $\lim_{y\rightarrow+\infty}S_2(y)=0$ and $\lim_{y\rightarrow+\infty}S_3(y)=0$. Therefore we must have $\lim_{y\rightarrow+\infty}S_1(y)=0$. Because $\beta \geq  0$ implies
\[
\lim\limits_{y\rightarrow+\infty}F\left[\frac{\alpha\psi(\sigma_i)+\beta\{y-\kappa(\mu_{11})\}}{\varphi(\sigma_1)}\right]\\
>0,
\]
we have $p_{11}=0$. By the same argument, we can prove that $p_{ij}=0$ for all $i,j$. Therefore $Q_{ij}(y)$ are linearly uncorrelated.

\item[(b)] $\beta<0$. We let $\tilde \beta=-\beta$, $\tilde y=-y$, $\tilde\mu_{ij}=-\mu_{ij}$, $\tilde\kappa(\tilde\mu_{ij})=-\kappa(-\mu_{ij})$, and 
\[
\tilde Q_{ij}(\tilde y)=Q_{ij}( y) =\phi\left(\frac{\tilde y-\tilde\mu_{ij}}{\sigma_i}\right)
F\left\{\frac{\alpha\psi(\sigma_i)+\tilde\beta(\tilde y-\tilde\kappa(\tilde\mu_{ij}))}{\varphi(\sigma_i)}\right\}.
\] 
Because $\tilde\beta>0$, $\{\tilde Q_{ij}(\tilde y): i=1,\ldots,I; j=1,\ldots,J_i  \}$  are linearly uncorrelated according the discussion in the case (a). So $\{  Q_{ij}(y): i=1,\ldots,I;j=1,\ldots,J_i \}$ are linearly uncorrelated.
\end{enumerate}
\end{proof}

\begin{proof}[Proof of Theorem \ref{thm:NF}]
We use a proof by contradiction to show that the parameters can be identified by the observed distribution $P(y, R=1)$.
Suppose that there were two sets of parameters satisfying the same observed distribution:
\begin{gather}\label{prf:1.1}
    \frac{1}{\sigma}\phi\left(\frac{y-\mu}{\sigma}\right)\cdot F(\alpha+\beta y)=
        \frac{1}{\sigma'}\phi\left(\frac{y-\mu'}{\sigma'}\right)\cdot F(\alpha'+\beta' y).
\end{gather}
Below we show the results (a), (b) and (c) of the theorem one by one.

\begin{enumerate}
\item[(a)] Equation (\ref{prf:1.1}) implies
\begin{gather}
\frac{\phi\left(\frac{y-\mu}{\sigma}\right)}{\phi\left(\frac{y-\mu'}{\sigma'}\right)}\cdot\frac{F(\alpha+\beta y)}{F(\alpha'+\beta' y)}=\frac{\sigma}{\sigma'}\in(0,+\infty).\label{prf:1.2}
\end{gather}
Lemma \ref{lem:1} implies that $\sigma\neq\sigma'$ is impossible since as $y\rightarrow \pm \infty$ the limit of (\ref{prf:1.2}) is either $0$ or $+\infty$. Thus we must have $\sigma=\sigma'$.
To prove $|\beta|=|\beta'|$, we discuss the following four cases.
\begin{enumerate}
\item[(a.1)] $\beta\geq 0, \beta'\geq 0$. 
We must have $\mu=\mu'$, because otherwise as $y \rightarrow +\infty$, the limit of (\ref{prf:1.2}) is either $0$ or $+\infty$ by Lemma \ref{lem:2}. Therefore, $F(\alpha+\beta y)=F(\alpha'+\beta' y)$ for any $y$. Because $F(\cdot)$ is strictly monotone, we must have 
$\alpha=\alpha'$ and $\beta=\beta'$.

\item[(a.2)] $\beta \geq 0, \beta' <0.$ 
Because $\lim_{y\rightarrow+\infty}F(\alpha+\beta y)$ is finite and positive, we have that
$$
\lim\limits_{y\rightarrow+\infty}\frac{\phi\left(\frac{y-\mu}{\sigma}\right)}{\phi\left(\frac{y-\mu'}{\sigma}\right)}\cdot\frac{1}{F(\alpha+\beta' y)}
$$
is finite and positive. 
According to Lemma \ref{lem:4}, we must have $\lim_{z\rightarrow-\infty}F(z)/e^{\delta z}=c\in(0,+\infty)$ with $\delta=(\mu-\mu')/(\sigma^2\beta')$. Let $y\rightarrow -\infty$, and a similar application of Lemma \ref{lem:4} to (\ref{prf:1.2}) gives us $\delta = (\mu'-\mu)/(\sigma^2\beta)$.
Therefore, we have $\beta=-\beta'$.

\item[(a.3)] $\beta<0, \beta'<0$. The discussion is similar to (a.1) by letting $y\rightarrow -\infty.$

\item[(a.4)] $\beta < 0, \beta ' \geq 0.$ The discussion is similar to (a.2).
\end{enumerate}

\item[(b)] From (a), $\sigma^2$ and $|\beta|$ are identifiable. When the sign of $\beta$ is known, we need only to consider the cases (a.1) or (a.3). We have proved identifiability of all the parameters in (a.1) and (a.3).

\item[(c)] 
We will prove that under the Condition \ref{condi:A}, the cases (a.2) and (a.4) are impossible. For case (a.2) with $\beta \geq 0, \beta' <0$, as $y\rightarrow +\infty$, the limit of $F(\alpha + \beta y)$ is finite and positive, and therefore the limit of
$$
\frac{\phi\left(\frac{y-\mu}{\sigma}\right)}{\phi\left(\frac{y-\mu'}{\sigma}\right)}\cdot\frac{1}{F(\alpha+\beta' y)}   
$$ 
is also finite and positive. Due to Lemma \ref{lem:4} and $\beta' < 0$, we must have $ \lim_{ z\rightarrow -\infty } F(z) / e^{ \delta z} = c\in (0, +\infty)$ for some $\delta>0$, which is contradictory to Condition \ref{condi:A}. Similarly, the case (a.4) is also impossible. For cases (a.1) and (a.3) we have already established the identifiability of the parameters. 
\end{enumerate}
\end{proof}

\begin{proof}[Proof of Corollary \ref{cor:NT}]
Suppose that two sets of parameters $(\sigma,\mu,\nu,\alpha,\beta)$ and $(\sigma',\mu',\nu',\alpha',\beta')$ satisfy the same observed distribution:
\begin{align*}
    \frac{1}{\sigma}\phi\left(\frac{y-\mu}{\sigma}\right)\cdot T_\nu(\alpha+\beta y)=&
        \frac{1}{\sigma'}\phi\left(\frac{y-\mu'}{\sigma'}\right)\cdot T_{\nu'}(\alpha'+\beta' y).
\end{align*}
Similar to the proof of Theorem \ref{thm:NF}(a), we can show $\sigma = \sigma'$ according to Lemma \ref{lem:1}. For any $\nu$, a $t$ random variable with degrees of freedom $\nu$ satisfies $t_\nu(z)\propto (1+z^2)^{-(1+\nu)/2}$, and thus $T_\nu(z)$ satisfies Condition \ref{condi:A} in Theorem \ref{thm:NF} by L'Hospital's rule. Therefore, cases (a.2) and (a.4) in the proof of Theorem \ref{thm:NF} are impossible according to the proof of Theorem \ref{thm:NF}(c). Below we discuss cases (a.1) and (a.3).

For case (a.1) with $\beta \geq 0, \beta'\geq 0$, we can show $\mu = \mu'$ similar to the proof of Theorem \ref{thm:NF} according to Lemma \ref{lem:2}. 
Because $(\mu, \sigma^2)$ are identifiable, we have $T_\nu(\alpha+\beta y) = T_{\nu'}(\alpha'+\beta' y) $ for any $y$. Under the condition $\beta\neq0$ required by Corollary \ref{cor:NT},  we can show that, as $y\rightarrow -\infty$, the limit of $T_\nu(\alpha+\beta y) / T_{\nu'}(\alpha'+\beta' y) $ is $0$ if $\beta'=0$. Therefore $\beta'>0$.  Then as $y\rightarrow -\infty$, the limit of $T_\nu(\alpha+\beta y) / T_{\nu'}(\alpha'+\beta' y) $ is $+\infty$ if $\nu<\nu'$, and $0$ if $\nu>\nu'$. Therefore $\nu=\nu'$, and thus $\alpha=\alpha', \beta=\beta'$.

The discussion for case (a.3) is similar to case (a.1).
\end{proof}

\begin{proof}[Proof of Theorem \ref{thm:NFx}]
Suppose that there were two sets of parameters satisfying the same observed distribution:
\begin{gather}
    \frac{1}{\sigma(x,\theta)}\phi\left\{\frac{y-\mu(x,\gamma)}{\sigma(x,\theta)}\right\}  F\{g(x,\alpha)+\beta y\}=
        \frac{1}{\sigma(x,\theta')}\phi\left\{\frac{y-\mu(x,\gamma')}{\sigma(x,\theta')}\right\}  F\{g(x,\alpha')+\beta' y\}. \label{prf:2.0}
\end{gather}
Replacing $\mu$, $\sigma^2$ and $\alpha$ in the proof of Theorem \ref{thm:NF} with $\mu(x,\gamma)$, $\sigma^2(x,\theta)$ and $g(x,\alpha)$ and conditional on $X=x$, the identifiability of $\mu(x,\gamma)$, $\sigma^2(x,\theta)$, $g(x,\alpha)$ and $\beta$ is the same as that of $\mu$, $\sigma^2$, $\alpha$ and $\beta$ in Theorem \ref{thm:NF}. As $x$ varies, we obtain identifiability of functions $\mu(\cdot,\gamma)$, $\sigma^2(\cdot,\theta)$ and $g(\cdot,\alpha)$.
Because there are one-to-one mappings between the parameters $(\gamma,\theta,\alpha)$ and these functions, the parameters are identifiable. Therefore, the results (a), (b) and (c) of Theorem \ref{thm:NFx} hold. 

Now we will prove the result (d). Similar to the proof of Theorem \ref{thm:NF},  we can show the identifiability of the parameters for cases (a.1) $\beta\geq 0, \beta'\geq 0$  and (a.3) $\beta< 0, \beta'< 0$. We need only to show that (a.2) $\beta\geq 0, \beta'<0$  and (a.4) $\beta< 0, \beta'\geq 0$  are impossible under the condition in (d) of Theorem \ref{thm:NFx}. Without loss of generality, we show case (a.2), and the discussion for case (a.4) is analogous.

Note that $\theta$ is identifiable and $\beta' = -\beta$ by the result (a). For case (a.2) with $\beta \geq 0$ and $ \beta' <0$, 
Equation (\ref{prf:2.0}) implies that
\begin{gather}
\lim\limits_{y\rightarrow \pm \infty}\frac{\phi\left\{\frac{y-\mu(x,\gamma)}{\sigma(x,\theta)}\right\}}{\phi\left\{\frac{y-\mu(x,\gamma')}{\sigma(x,\theta)}\right\}}\cdot \frac{F\{g(x,\alpha)+\beta y\}}{F\{g(x,\alpha')+\beta' y\}}=1.
\end{gather}
We let $y\rightarrow + \infty$ and apply Lemma \ref{lem:4}; we then let $y\rightarrow - \infty$ and apply Lemma \ref{lem:4} again. Consequently, we must have two sets of conditions, i.e.,
$\lim_{z\rightarrow-\infty}F(z)/e^{\delta z}=c$ with
\begin{eqnarray}
\delta&=&-\frac{\mu(x,\gamma)-\mu(x,\gamma')}{\sigma^2(x,\theta)\beta}=\frac{\mu(x,\gamma)-\mu(x,\gamma')}{\sigma^2(x,\theta)\beta'},\label{prf:2.1}\\
\log(c)&=&\frac{\mu(x,\gamma)^2-\mu(x,\gamma')^2}{2\sigma(x,\theta)^2}-\delta g(x,\alpha)=\frac{\mu(x,\gamma')^2-\mu(x,\gamma)^2}{2\sigma(x,\theta)^2}-\delta g(x,\alpha').\label{prf:2.4}
\end{eqnarray}
Equation (\ref{prf:2.4}) holds for any $x$ and therefore holds for a particular $x_0$. Taking the difference between Equation (\ref{prf:2.4}) for $x$ and $x_0$, we have that
\begin{eqnarray*}
-\delta\{g(x,\alpha)-g(x_0,\alpha)\}  & = & \frac{\mu^2(x,\gamma')-\mu^2(x,\gamma)}{2\sigma^2(x,\theta)}-\frac{\mu^2(x_0,\gamma')-\mu^2(x_0,\gamma)}{2\sigma^2(x_0,\theta)},\\
\delta\{g(x,\alpha')-g(x_0,\alpha')\} & = & \frac{\mu^2(x,\gamma')-\mu^2(x,\gamma)}{2\sigma^2(x,\theta)}-\frac{\mu^2(x_0,\gamma')-\mu^2(x_0,\gamma)}{2\sigma^2(x_0,\theta)}.
\end{eqnarray*}

From the first identity of (\ref{prf:2.1}), we have $\mu(x,\gamma')=\mu(x,\gamma)+\delta\beta\sigma^2(x,\theta)$. Plugging it into the above two equations, we have 
\begin{eqnarray}
-g(x,\alpha)+g(x_0,\alpha) & = & \beta\{\mu(x,\gamma)-\mu(x_0,\gamma)\}+\frac{\delta\beta^2}{2}\{\sigma^2(x,\theta)-\sigma^2(x_0,\theta)\},\label{prf:2.2}\\
g(x,\alpha')-g(x_0,\alpha') & = & \beta\{\mu(x,\gamma)-\mu(x_0,\gamma)\}+\frac{\delta\beta^2}{2}\{\sigma^2(x,\theta)-\sigma^2(x_0,\theta)\}.\label{prf:2.3}
\end{eqnarray}
Either one of (\ref{prf:2.2}) and (\ref{prf:2.3}) conflicts with the condition that $a\cdot\mu(x,\gamma)+b\cdot \sigma^2(x,\theta)+  g(x,\alpha) \neq c$ for any nonzero vector $(a,b,c)$ and for any $(\gamma,\theta,\alpha)$. Thus (a.2) is impossible.
\end{proof}

\begin{proof}[Proof of Theorem \ref{thm:slope}]
Theorem \ref{thm:slope} is a special case of Theorem \ref{thm:NFx} with $\sigma(x,\theta) = \sigma$ being a constant. The first identity of Equation (\ref{prf:2.1}) implies that for any $x$ and a particular $x_0$,
\[-\frac{\mu(x,\gamma_1)-\mu(x,\gamma_1')}{\sigma^2\beta}=-\frac{\mu(x_0,\gamma_1)-\mu(x_0,\gamma_1')}{\sigma^2\beta}.\]
Therefore, we have
$
\mu(x,\gamma_1)-\mu(x_0,\gamma_1)=\mu(x,\gamma_1')-\mu(x_0,\gamma_1'),
$
implying
$
 \partial \mu(x,\gamma_1) / \partial x =  \partial \mu(x,\gamma_1') / \partial x.
$
Because $\sigma(x,\theta) = \sigma$ is a constant,
Equations (\ref{prf:2.2}) and (\ref{prf:2.3}) imply
$
g(x,\alpha')-g(x_0,\alpha')=-g(x,\alpha)+g(x_0,\alpha),
$
and therefore,
$
 \partial g(x,\alpha') / \partial x =-  \partial g(x,\alpha) / \partial x .
$
\end{proof}

\begin{proof}[Proof of Corollary \ref{cor:logit}]
The models of Corollary \ref{cor:logit} are special cases of Theorem \ref{thm:NFx} with $\sigma(x,\theta) = \sigma$ being a constant, $\mu(x,\gamma)=\gamma_0+x^T\gamma_1$, and $g(x,\alpha)=\alpha_0+x^T\alpha_1$. Below we show the results (a) and (b) one by one.

\begin{enumerate}
\item[(a)]
Similar to the proof of Theorem \ref{thm:NFx}, we need only to show that cases (a.2) $\beta\geq 0, \beta'< 0$  and (a.4) $\beta< 0, \beta'\geq 0$ are impossible. Without loss of generality, we discuss only case (a.2). 
If there were two different sets of parameters satisfying the same observed distribution, we obtain from Equation (\ref{prf:2.4}) that
\begin{eqnarray}
\gamma_0'=\gamma_0+\delta\sigma^2\beta, \quad 
\alpha_0' =-\alpha_0-2\log(c)/\delta, \quad 
\gamma_1'=\gamma_1,\quad 
\beta'=-\beta,\quad
\alpha_1'=-\alpha_1, 
\label{eq::logit-01}\\
\alpha_1+\beta\gamma_1=0, \quad 2\delta\alpha_0+2\delta\beta\gamma_0+\delta^2\sigma^2\beta^2+2\log(c)=0.
\label{eq::logit-02}
\end{eqnarray}
The above equations imply that $\tau_1=0$ and $\tau_2=0$, which conflicts with the condition that $\tau_1\neq0$ or $\tau_2\neq0$. So the case (a.2) is impossible. The proof of (a.4) is similar. Therefore, $\tau_1\neq0$ or $\tau_2\neq0$ is a sufficient condition for identifiability of the parameters.

\item[(b)] For the Logistic missing mechanism, we can further prove the necessity of the condition $\tau_1\neq0$ or $\tau_2\neq0$. If $\tau_1=0$ and $\tau_2=0$, we can verify the following equation:
$$
\exp\left\{-\frac{(y-\gamma_0-x\gamma_1)^2}{2\sigma^2}\right\}\frac{\exp(\alpha_0+x\alpha_1+\beta y)}{1+\exp(\alpha_0+x\alpha_1+\beta y)} 
=
\exp\left\{-\frac{(y-\gamma_0'-x\gamma_1')^2}{2\sigma^2}\right\}\frac{\exp(\alpha_0'+x\alpha_1'+\beta' y)}{1+\exp(\alpha_0'+x\alpha_1'+\beta' y)},
$$
i.e., the two sets of parameters satisfying Equations (\ref{eq::logit-01}) and (\ref{eq::logit-02}) must also have the same observed distribution if $\tau_1=0$ and $\tau_2=0$. Therefore $\tau_1\neq0$ or $\tau_2\neq0$ is also a necessary condition for identifiability of the parameters.
\end{enumerate}
\end{proof}

\begin{proof}[Proof of Theorem \ref{thm:MNF} and \ref{thm:MNCondF}]
Theorem \ref{thm:MNF} is the special case of Theorem \ref{thm:MNCondF} with $\psi(\cdot)=1$, $\kappa(\cdot)=0$ and $\varphi(\cdot)=1$. Therefore, we prove only Theorem \ref{thm:MNCondF}.

For notational convenience, we use double indices $(i, j)$ for $k$ such that $\{\sigma_i^2 : i = 1,\ldots,I\}$ are sorted in a decreasing order, and then $\{\mu_{ij} : j = 1,...,J_i; \sum_{i=1}^I J_i = K\}$ are sorted in a decreasing order for each $i$.
We first define:
\begin{eqnarray*}
Q_{ij}(y) &=&
\frac{\pi_{ij}}{\sigma_i}\phi\left(  \frac{y-\mu_{ij}}{\sigma_i}\right) F\left[\frac{\alpha\psi(\sigma_i)+\beta\{y-\kappa(\mu_{ij})\}}{\varphi(\sigma_i)}\right],\\
h_1(y)&=&
\sum\limits_{i=1}^{I}\sum\limits_{j=1}^{J_i} Q_{ij}(y)/Q_{11}(y),\quad 
h_2(y)=\sum\limits_{i=1}^{I}\sum\limits_{j=1}^{J_i} Q_{ij}(y)/Q_{1J_1}(y).
\end{eqnarray*}
Suppose that there were another set of parameters having the same observed distribution, and we use $Q_{ij}'(y)$, $h_1'(y)$ and $h_2'(y)$ to denote the  functions under this set of parameters.
We have
\begin{gather}
\sum\limits_{i=1}^{I}\sum\limits_{j=1}^{J_i} Q_{ij}(y) = \sum\limits_{i=1}^{I'}\sum\limits_{j=1}^{J_i'} Q_{ij}'(y).\label{prf:6.1} 
\end{gather}

By the definitions of $h_1(y),h_2(y),h_1'(y)$ and $h_2'(y)$, we have
\begin{equation}
 \frac{Q_{11}(y)}{Q_{11}'(y)}\cdot\frac{h_1(y)}{h_1'(y)}=1,
\quad 
\frac{Q_{1J_1}(y)}{Q_{1J_1'}'(y)}\cdot\frac{h_2(y)}{h_2'(y)}=1. \label{eq::def-h1-h2}
\end{equation}
We can re-express $h_1(y)$ as
\begin{align*}
h_1(y)=1+&\sum\limits_{j=2}^{J_1}\frac{ \pi_{1j}}{ \pi_{11}}\exp\left\{\frac{(\mu_{1j}-\mu_{11})y}{2\sigma_1^2}+\frac{\mu_{11}^2-\mu_{1j}^2}{2\sigma_1^2}\right\}\frac{F\left[\frac{\alpha\psi(\sigma_1)+\beta\{y-\kappa(\mu_{1j})\}}{\varphi(\sigma_1)}\right]}{F\left[\frac{\alpha\psi(\sigma_1)+\beta\{y-\kappa(\mu_{11})\}}{\varphi(\sigma_1)}\right]}\\
+&\sum\limits_{i=2}^{I}\sum\limits_{j=1}^{J_i}\frac{ \pi_{ij}\sigma_1}{ \pi_{11}\sigma_i}\exp\left\{-\frac{(y-\mu_{ij})^2}{2\sigma_i^2}+\frac{(y-\mu_{11})^2}{2\sigma_1^2}\right\}\frac{F\left[\frac{\alpha\psi(\sigma_i)+\beta\{y-\kappa(\mu_{ij})\}}{\varphi(\sigma_i)}\right]}{F\left[\frac{\alpha\psi(\sigma_1)+\beta\{y-\kappa(\mu_{11})\}}{\varphi(\sigma_1)}\right]}.
\end{align*}
Note  that $\varphi(\cdot)$ and $\kappa(\cdot)$ are increasing functions, and therefore the limit
\[
\lim\limits_{y\rightarrow+\infty}\frac{F\left[\frac{\alpha\psi(\sigma_i)+\beta\{y-\kappa(\mu_{ij})\}}{\varphi(\sigma_i)}\right]}{F\left[\frac{\alpha\psi(\sigma_1)+\beta\{y-\kappa(\mu_{11})\}}{\varphi(\sigma_1)}\right]}
\quad (i=1, \ldots, I)
\]
must be finite and positive regardless of the sign of $\beta$. The exponential terms in $h_1(y)$ converge to $0$ as $y\rightarrow+\infty$,
so we have $\lim_{y\rightarrow+\infty}h_1(y)=1$. Similarly, $\lim_{y\rightarrow+\infty}h_1'(y)=1$, $\lim_{y\rightarrow-\infty}h_2(y)=1$, and $\lim_{y\rightarrow-\infty}h_2'(y)=1$. From Equation (\ref{eq::def-h1-h2}), we have that
\begin{eqnarray}
\lim\limits_{y\rightarrow+\infty}\frac{Q_{11}(y)}{Q_{11}'(y)}
&=&
\lim\limits_{y\rightarrow+\infty}\frac{ \pi_{11}\sigma_1'\phi\left(\frac{y-\mu_{11}}{\sigma_1}\right)}{ \pi_{11}'\sigma_1\phi\left(\frac{y-\mu_{11}'}{\sigma_1'}\right)}
\cdot
\frac{F\left[\frac{\alpha\psi(\sigma_1)+\beta\{y-\kappa(\mu_{11})\}}{\varphi(\sigma_1)}\right]}{F\left[\frac{\alpha'\psi(\sigma_1')+\beta'\{y-\kappa(\mu_{11}')\}}{\varphi(\sigma_1')}\right]}= 1,\label{prf:6.2}  \\
\lim\limits_{y\rightarrow-\infty}\frac{Q_{1J_1}(y)}{Q_{1J_1'}'(y)}
&=&
\lim\limits_{y\rightarrow-\infty}\frac{ \pi_{1J_1}\sigma_1'\phi\left(\frac{y-\mu_{1J_1}}{\sigma_1}\right)}{ \pi_{1J_1'}'\sigma_1\phi\left(\frac{y-\mu_{1J_1'}'}{\sigma_1'}\right)}
\cdot
\frac{F\left[\frac{\alpha\psi(\sigma_1)+\beta\{y-\kappa(\mu_{1J_1})\}}{\varphi(\sigma_1)}\right]}{F\left[\frac{\alpha'\psi(\sigma_1')+\beta'\{y-\kappa(\mu_{1J_1'}')\}}{\varphi(\sigma_1')}\right]}= 1.\label{prf:6.3}
\end{eqnarray} 
If $\sigma_1\neq\sigma_1'$, Lemma \ref{lem:1} implies that the limits in (\ref{prf:6.2}) and (\ref{prf:6.3}) must be $0$ or $+\infty$, which contradict with the limits above. Therefore, we must have $\sigma_1=\sigma_1'$. We divide the remaining discussion into the following three cases.

\begin{enumerate}
\item[(a)] 
$\beta > 0$. 
We first observe that 
\[
\lim_{y\rightarrow +\infty} F\left[\frac{\alpha\psi(\sigma_1)+\beta\{y-\kappa(\mu_{11})\}}{\varphi(\sigma_1)}\right]
\] 
is finite and positive. We apply Lemma \ref{lem:4} to Equation (\ref{prf:6.2}) and obtain $\mu_{11}=\mu_{11}'$ and $\beta'\geq0$, because Condition \ref{condi:A} rules out the second case in Lemma \ref{lem:4}. We further demonstrate that $\beta' = 0$ is also impossible. Otherwise, Equation (\ref{prf:6.3}) and Lemma \ref{lem:4} with $y\rightarrow -\infty$ will conflict with Condition \ref{condi:A}. With $\sigma_1=\sigma_1'$, $\mu_{11} = \mu_{11}'$, $\beta > 0$ and $\beta ' >0$, we must have $\pi_{11} = \pi_{11}'$ by Equation (\ref{prf:6.2}).

Let $\theta_0=\alpha\psi(\sigma_1)/\varphi(\sigma_1)-\beta\kappa(\mu_{11})$, $\theta_1=\beta /\varphi(\sigma_1)$, $\theta_0'=\alpha'\psi(\sigma_1)/\varphi(\sigma_1)-\beta'\kappa(\mu_{11})$, and $\theta_1'=\beta' /\varphi(\sigma_1)$.
We plus both sides of  Equation (\ref{prf:6.1}) by $-\pi_{11}/\sigma_1 \cdot\phi \{(y-\mu_{11})/\sigma_1\}h_1(y)F(\theta_0'+\theta_1' y)$,  and obtain that
    \[  
    \frac{\frac{\pi_{11}}{\sigma_1}\phi\left(  \frac{y-\mu_{11}}{\sigma_1}\right) \{h_1(y)F(\theta_0+\theta_1 y)-h_1(y)F(\theta_0'+\theta_1' y)\}}{\frac{\pi_{11}}{\sigma_1}\phi\left(  \frac{y-\mu_{11}}{\sigma_1}\right) \{h_1'(y)F(\theta_0'+\theta_1' y)-h_1(y)F(\theta_0'+\theta_1' y)\}} = 1,\]
and thus 
\begin{eqnarray}
\lim_{y\rightarrow+\infty}\frac{F(\theta_0+\theta_1 y)-F(\theta_0'+\theta_1' y)}{h_1'(y)-h_1(y)}
\cdot  \frac{h_1(y)}{F(\theta_0'+\theta_1' y)}=1.
\label{eq::thm6-limit}
\end{eqnarray}
The second fraction of the above equation converges to one. By the definitions of $h_1(y)$ and $h_1'(y)$,  there exist $\delta_1,\delta_2\geq0$ satisfying $\delta_1+\delta_2>0$ such that 
\[h_1'(y)-h_1(y)=O(e^{-\delta_1 y^2-\delta_2y}),\quad y\rightarrow+\infty.\] 
This leads to a contradiction with Condition \ref{condi:B}, implying that Equation (\ref{eq::thm6-limit}) is impossible. So we must have  $\beta=\beta'$ and $\alpha=\alpha'$. Equation (\ref{prf:6.1}) implies that 
$$
\{ Q_{ij}(y) : i=1,\ldots, I; j=1,\ldots, J_i   \}  \cup  \{  Q_{ij}'(y) : i=1,\ldots, I' ; j=1,\ldots, J_i'  \}
$$ 
are linearly correlated, and by Lemma \ref{lem:5} we must have that $\sigma_i=\sigma_i'$, $\mu_{ij}=\mu_{ij}'$, $I=I'$ and $J_i=J_i'$ for all $i,j$.

\item[(b)] $\beta = 0$.  Similar to the argument of case $\beta > 0$, we first apply Lemma \ref{lem:4} to obtain $\mu_{11}=\mu_{11}'$ and $\beta' \geq  0$. We then rule out the case $\beta' > 0$, otherwise Equation (\ref{prf:6.3}) contradicts Condition \ref{condi:A}. By integrating over $y$ in both sides of Equation (\ref{prf:6.1}), we obtain 
$$
\sum_{i=1}^I \sum_{j=1}^{J_i} \pi_{ij}F\{\alpha\psi(\sigma_i)/\varphi(\sigma_i)\}
=
\sum_{i=1}^{I'} \sum_{j=1}^{J_i'} \pi_{ij}'F\{\alpha'\psi(\sigma_i')/\varphi(\sigma_i')\} .
$$
Dividing both sides of  Equation (\ref{prf:6.1}) by the left hand side of the above identity, we obtain two normal mixtures leading to the same observed distribution. Applying the identifiability of the normal mixture distribution \citep[][Theorem 3.1.2]{titterington1985statistical}, we obtain that $\sigma_i^2=\sigma_i'^2,\mu_{ij}=\mu_{ij}'$ and $\pi_{ij}F\{\alpha\psi(\sigma_i)/\varphi(\sigma_i)\}=\pi_{ij}'F\{\alpha'\psi(\sigma_i)/\varphi(\sigma_i)\}$, for all $i$ and $j$. Note that
$\psi(\cdot)$ and $\varphi(\cdot)$ are positive. If $\alpha > \alpha'$, we have 
$
F\{\alpha\psi(\sigma_i)/\varphi(\sigma_i)\} > F\{\alpha'\psi(\sigma_i)/\varphi(\sigma_i)\}
$
and $\pi_{ij} < \pi_{ij}'$ for all $i$ and $j$, which contradicts  $\sum_{i=1}^I \sum_{j=1}^{J_i}  \pi_{ij} =\sum_{i=1}^{I'} \sum_{j=1}^{J_i'} \pi_{ij}' = 1.$ Similarly, $\alpha < \alpha'$ is also impossible. Therefore, we have $\alpha = \alpha'$, and thus $\pi_{ij} = \pi_{ij}'$ for all $i$ and $j$.

\item[(c)] $\beta<0$. Let $\tilde \beta=-\beta$, $\tilde y=-y$, $\tilde\mu_k=-\mu_k$ and  $\tilde\kappa(\tilde \mu_{k})=-\kappa(-\mu_{k})$. We can use the same argument of (a) to prove identifiability of the parameters.
\end{enumerate}
\end{proof}

\begin{proof}[Proof of Theorem \ref{thm:TF}]
Suppose that there were two different sets of parameters satisfying the same observed distribution:
\begin{align}
    \sum\limits_{k=1}^{K} \frac{\pi_{k}}{\omega} t_{\nu}\left(\frac{y-\mu_{k}}{\omega}\right)F(\alpha+\beta y)
    =& \sum\limits_{k=1}^{K'} \frac{\pi_{k}'}{\omega'}  t_{\nu'}\left(\frac{y-\mu_{k}'}{\omega'}\right)F(\alpha'+\beta' y).\label{prf:4.1}
\end{align}

We first show that the degrees of freedom $\nu$ is identifiable. Suppose $\nu<\nu'$. We divide both sides of Equation (\ref{prf:4.1}) by
$1/\omega \cdot t_{\nu}\{(y-\mu_{1})/\omega\}$, and obtain that
\begin{eqnarray*}
\frac{\sum\limits_{k=1}^{K} \pi_{k}t_{\nu}\left(\frac{y-\mu_{k}}{\omega}\right)}{t_{\nu}\left(\frac{y-\mu_{1}}{\omega}\right)}F(\alpha+\beta y)
=
\frac{ \sum\limits_{k=1}^{K'} \pi_{k}'\frac{\omega}{\omega'}  t_{\nu'}\left(\frac{y-\mu_{k}'}{\omega'}\right)}{t_{\nu}\left(\frac{y-\mu_{1}}{\omega}\right)}F(\alpha'+\beta' y) .
\end{eqnarray*}
If $\beta\geq0$, we let $y\rightarrow+\infty$, otherwise we let $y\rightarrow-\infty$.
The left hand side of the above equation converges to a positive constant larger than $\pi_1 F(\alpha)$, but the right hand side  converges to zero. Thus $\nu<\nu'$ is impossible. Similarly, $\nu>\nu'$ is also impossible. We must have $\nu=\nu'$.

We then show that other parameters are identifiable. If $\beta>0, \beta'\leq0$, we let $y \rightarrow -\infty$. The left hand side of the above equation converges to zero, but the right hand side converges to a positive constant.
Therefore, this case is impossible. 
Similarly, we can prove that the other three cases with different signs,
$(\beta<0,\beta'\geq0)$,  $(\beta\leq 0, \beta' >0)$ and $(\beta\geq 0, \beta' <0)$, are impossible. Therefore, $\beta$ and $\beta'$ must have the same sign. We then need only to discuss the following three cases.

\begin{enumerate}
\item[(a)] $\beta>0, \beta'>0.$ 
Suppose that $(\beta,\alpha, \omega )\neq(\beta',\alpha', \omega')$. Adding $-\sum_{k=1}^{K} \pi_{k}/\omega\cdot t_{\nu}\{(y-\mu_{k})/\omega\}F(\alpha'+\beta' y)$ to both sides of (\ref{prf:4.1}), and then dividing them by $1/\omega\cdot t_{\nu}\{(y-\mu_{1})/\omega\}$, we have 
\begin{gather}
\frac{\sum\limits_{k=1}^{K} \pi_{k}t_{\nu}\left(\frac{y-\mu_{k}}{\omega}\right)}{t_{\nu}\left(\frac{y-\mu_{1}}{\omega}\right)}\{F(\alpha+\beta y)-F(\alpha'+\beta' y)\}
=
\frac{ \sum\limits_{k=1}^{K'} \pi_{k}'\frac{\omega}{\omega'}  t_{\nu}\left(\frac{y-\mu_{k}'}{\omega'}\right)-\sum\limits_{k=1}^{K} \pi_{k} t_{\nu}\left(\frac{y-\mu_{k}}{\omega}\right)}{t_{\nu}\left(\frac{y-\mu_{1}}{\omega}\right)}F(\alpha'+\beta' y) . 
\label{eq::t-identi}
\end{gather} 
If $\omega \neq \omega'$, as $y\rightarrow+\infty$, the left hand side of the above equation converges to zero, but the right hand side converges to a nonzero constant.  Therefore, we must have $\omega=\omega'$.

We can prove that, there exists an $M>0$ such that 
\[
\frac{ \sum\limits_{k=1}^{K'} \pi_{k}'  t_{\nu}\left(\frac{y-\mu_{k}'}{\omega'}\right)-\sum\limits_{k=1}^{K} \pi_{k} t_{\nu}\left(\frac{y-\mu_{k}}{\omega}\right)}{t_{\nu}\left(\frac{y-\mu_{1}}{\omega}\right)} = O(y^{-M}) \quad
(y\rightarrow\pm \infty).
\]
Then Equation (\ref{eq::t-identi}) implies that the following limits are both finite and positive:
$$
\lim_{y\rightarrow+\infty}  y^{M}  \{F(\alpha+\beta y)-F(\alpha'+\beta' y)\},\quad 
\lim_{y\rightarrow-\infty}   y^{M}   \{1-F(\alpha+\beta y)/F(\alpha'+\beta' y)\}  .
$$ 
Let $y'=\alpha'+\beta' y$, substitute $y$ with $y'$ in the above limits, and we find that the limits contradict Condition \ref{condi:C}.
So we must have $\beta=\beta'$ and $\alpha=\alpha'$. Therefore, $F(\alpha+\beta y)$ and $F(\alpha'+\beta' y)$ cancel each other in Equation (\ref{prf:4.1}). The identifiability of other parameters reduces to the identifiability of location mixture of $t$ distributions \citep[][Theorem 3.1.2]{titterington1985statistical}.

\item[(b)]
$\beta<0, \beta'<0$. 
Define  $\tilde y=-y$, $\tilde\beta=-\beta$ and $\tilde\mu_{k}=-\mu_{k}$. Similar to the above case, we can prove that all the parameters are identifiable.

\item[(c)]
$\beta=\beta'=0$. 
Integrating both sides of Equation (\ref{prf:4.1}) over $y$, we obtain $F(\alpha) = F(\alpha')$, which implies $\alpha=\alpha'$ by strict monotonicity of $F(\cdot)$.  Dividing both sides of (\ref{prf:4.1}) by $F(\alpha)$, then the two sides are both mixtures of $t$ distributions. Therefore, $\mu_k'=\mu_k,\omega'=\omega$ and $\pi_{k} =\pi_{k}' $ by the identifiability of $t$ mixture distributions \citep[][Theorem 3.1.2]{titterington1985statistical}. 
\end{enumerate}
\end{proof}

\end{document}